\documentclass[12pt]{article}
\usepackage[hmarginratio=1:1,margin=2.5cm,top=32mm,columnsep=20pt]{geometry}
\usepackage{amsfonts}
\usepackage{amssymb}
\usepackage{amsthm}
\usepackage{amsmath}
\usepackage{latexsym}
\usepackage{color}  
\usepackage{graphicx}
\usepackage{float}
\usepackage{subfig}
\usepackage{caption}
\usepackage{epstopdf}
\usepackage{hyperref}

\begin{document}
\newtheorem{lemma}{Lemma}
\newtheorem{pron}{Proposition}
\newtheorem{thm}{Theorem}
\newtheorem{Corol}{Corollary}
\newtheorem{exam}{Example}
\newtheorem{defin}{Definition}
\newtheorem{remark}{Remark}
\newtheorem{property}{Property}
\newcommand{\la}{\frac{1}{\lambda}}
\newcommand{\sectemul}{\arabic{section}}
\renewcommand{\theequation}{\sectemul.\arabic{equation}}
\renewcommand{\thepron}{\sectemul.\arabic{pron}}
\renewcommand{\thelemma}{\sectemul.\arabic{lemma}}
\renewcommand{\thethm}{\sectemul.\arabic{thm}}
\renewcommand{\theCorol}{\sectemul.\arabic{Corol}}
\renewcommand{\theexam}{\sectemul.\arabic{exam}}
\renewcommand{\thedefin}{\sectemul.\arabic{defin}}
\renewcommand{\theremark}{\sectemul.\arabic{remark}}
\newcommand{\tnewroman}{\fontspec{Times New Roman}}
\def\REF#1{\par\hangindent\parindent\indent\llap{#1\enspace}\ignorespaces}
\def\lo{\left}
\def\ro{\right}
\def\be{\begin{equation}}
\def\ee{\end{equation}}
\def\beq{\begin{eqnarray*}}
\def\eeq{\end{eqnarray*}}
\def\bea{\begin{eqnarray}}
\def\eea{\end{eqnarray}}
\def\d{\Delta_T}
\def\r{random walk}
\def\o{\overline}
\title {\Large{\bf Extremes of vector-valued processes by finite dimensional models}}
\author{Hui Xu$^{1}$ ~~ Mircea D. Grigoriu$^{1,2}$
\\ \small \href{hx223@cornell.edu}{hx223@cornell.edu} ~~\href{mdg12@cornell.edu}{mdg12@cornell.edu}\\
{\footnotesize\it 1. Center for Applied Mathematics, Cornell University, Ithaca, NY, USA}
\\ {\footnotesize\it 2. Department of Civil and Environmental Engineering, Cornell University, Ithaca, NY, USA}
}
\date{}
\maketitle

\begin{abstract}
Finite dimensional (FD) models, i.e., deterministic functions of time/space and finite sets of random variables, are constructed for target vector-valued random processes/fields.  They are required to have two properties. First, standard Monte Carlo algorithms can be used to generate their samples, referred to as FD samples. Second, under some conditions specified by several theorems, FD samples  can be used to estimate distributions of extremes and other functionals of target random functions.  Numerical illustrations involving two-dimensional random processes and apparent properties of random microstructures are presented to illustrate the implementation of FD models for these stochastic problems and show that they are accurate if the conditions of our theorems are satisfied.
\end{abstract}



\section{Introduction}
\setcounter{thm}{0}\setcounter{Corol}{0}\setcounter{lemma}{0}\setcounter{pron}{0}\setcounter{equation}{0}
\setcounter{remark}{0}\setcounter{exam}{0}\setcounter{property}{0}

Extremes of random functions are used extensively in science and engineering to, e.g., predict the intensity of natural hazards and their consequences, design physical systems, and characterize material microstructures.  Analytical expressions of the distributions of extremes of random functions are available in special cases of limited practical use.  Numerical methods need to be employed to find these distributions.  Their implementation requires to represent the random functions under consideration, referred to as target random functions, by finite dimensional (FD) models, i.e., deterministic functions of time/space and finite sets of random variables.  In contrast to target functions which, generally, have infinite stochastic dimensions as uncountable families of random variables indexed by time/space arguments, FD models have finite stochastic dimensions equal to the numbers of random variables in their definitions.

Our objective is to construct FD models for target random functions which have the following two properties.  First,
large sets of samples of FD models, referred to as FD samples, can be generated efficiently by standard Monte Carlo algorithms.  Second, distributions of extremes of target random functions can be estimated from FD samples to any desired accuracy.  We establish conditions under which the FD models under consideration have these two properties.

Three examples are presented to illustrate the construction of FD models and assess the accuracy of the estimates of extremes obtained from FD samples.  The first example is a two-dimensional random process whose components are linear forms of two independent non-Gaussian processes (Example~\ref{exam1}).  The process of the second example is defined by the joint solution of two linear oscillators subjected to scaled versions of the same non-Gaussian input process (Example~\ref{exam2}).  The third example estimates the apparent conductivity of a two-dimensional  material specimen with random properties (Example~\ref{exam3}).

The paper is organized as follows.  The FD models for vector-valued processes considered in this study are discussed
in the subsequent section (Sect.~2).  Our main result is in Sect.~3.  Conditions are established under which extremes of FD models converge weakly to those of target processes as their stochastic dimensions increase indefinitely.  Numerical illustrations are in Sect.~4.  Final comments can be found in Sect.~5.

\section{Finite dimension (FD) models}
\setcounter{thm}{0}\setcounter{Corol}{0}\setcounter{lemma}{0}\setcounter{pron}{0}\setcounter{equation}{0}
\setcounter{remark}{0}\setcounter{exam}{0}\setcounter{property}{0}

Throughout this paper, we denote the $L_2$  and $L_{\infty}$ norms by  $||\cdot||_2$ and $||\cdot||$.  For
$n$-dimensional vectors $x\in\mathbb{R}^n$, they have the expressions $||x||_2=(\sum_{i=1}^{l}x_i^2)^{1/2}$ and
$||x||=\max_{1\le i \le l} |x_i|$. These norms are equivalent since we deal with matrices and vectors in finite dimensional linear spaces \cite[Chap.~3]{R1987}.

Let $X(t)=(X_1(t),\cdots,X_n(t))^T, t\in[0,\tau]$, $0<\tau<\infty$, be an $\mathbb{R}^n$-valued process defined on a probability space $\big(\Omega,{\cal F},P\big)$ with means $\mu_i(t)=E[X_i(t)]=0$ and continuous correlation functions $c_{ij}(s,t)=E[X_i(s)X_j(t)]$, $i,j=1,\cdots,n$, $t,s\in[0,\tau]$.   The assumption $\mu_i(t)=0$ is not restrictive since, if $\mu_i(t)\not=0$, the deterministic function $\mu(t)=(\mu_1(t),\cdots,\mu_n(t))^T$ can be added to the samples of $X(t)$.

Denote by $\{\lambda_{i,k}\}$ and $\{\varphi_{i,k}\}$ the eigenvalues and the eigenfunctions of the correlation functions $\{c_{ii}(s,t)=E[X_i(s)\,X_i(t)\}$ of the components of $X(t)$, i.e., the solution of the integral equations
$\int_{0}^{\tau}c_{ii}(s,t)\varphi_{i,k}(t)dt=\lambda_{i,k}\varphi_{i,k}(s)$, $s\in[0,\tau]$ for $k=1,2,\cdots$ and $i=1,\cdots,n$.  Since the correlation functions of $X(t)$ are continuous, the series
$c_{ii}(s,t)=\sum_{k=1}^{\infty}\lambda_{i,k}\varphi_{i,k}(s)\varphi_{i,k}(t)$ converge absolutely and uniformly $L_2([0,\tau]^2)$ for all $i=1,\cdots,n$ by Mercer's Theorem \cite{J1909}.

Consider the family of $\mathbb{R}^n$-valued FD models  $\{X_d(t)=(X_{1,d_1}(t),\cdots,X_{1,d_n}(t))\}$ of $X(t)$ defined by
\begin{equation}\label{201b}
  X_{i,d_i}(t)=\sum_{k=1}^{d_i} Z_{i,k}\,\varphi_{i,k}(t),\quad d_i=1,2,\ldots, \quad i=1,\cdots,n, \quad t\in[0,\tau],
\end{equation}
where $d_i\geq 1$ are integers, $d=(d_1,\cdots,d_n)$, $\varphi_{i,k}(t)$, $k=1,\cdots,d_i$, are the top $d_i$ eigenfunctions of $c_{ii}(s,t)$, i.e., the eigenfunctions corresponding to the largest $d_i$ eigenvalues of $c_{ii}(s,t)$, and $\{Z_{i,k}\}$ are random coefficients whose samples are obtained by projecting samples $X_i(t,\omega)$ of $X_i(t)$ on the basis functions $\{\varphi_{i,k}(t)\}$, i.e.,
\begin{eqnarray}\label{202}
Z_{i,k}(\omega)=\int_{0}^{\tau}X_i(t,\omega)\varphi_{i,k}(t)dt,  \quad k=1,2,\cdots, \quad i=1,\cdots,n, \quad \omega\in\Omega.
\end{eqnarray}
The first two moments of the random variables $\{Z_{i,k}\}$ are
\begin{eqnarray*}
E[Z_{i,k}]=E\bigg[\int_0^\tau X_i(t)\,\varphi_{i,k}(t)\,dt\bigg]=\int_0^\tau E[X_i(t)] \,\varphi_{i,k}(t)\,dt=0
\end{eqnarray*}
and
\begin{eqnarray*}
  E[Z_{i,k}Z_{i,l}]&=&E\bigg[\int_{[0,\tau]^2} X_i(s)X_i(t)\varphi_{i,k}(s)\varphi_{i,l}(t)dsdt\bigg]
  =\int_{[0,\tau]^2} E[X_i(s)X_i(t)]\varphi_{i,k}(s)\varphi_{i,l}(t)dsdt\nonumber\\
   &=&\int_0^\tau \varphi_{i,k}(s)\bigg[\int_0^\tau c_{ii}(s,t)\varphi_{i,l}(t)dt\bigg]ds
   =\lambda_{i,l}\int_0^\tau \varphi_{i,k}(s)\varphi_{i,l}(s)ds=\lambda_{i,l}\delta_{kl},
\end{eqnarray*}
where the change of order of integration holds by Fubini's theorem.  The latter equality holds by the orthonormality of the eigenfunctions, i.e.,  $\langle \varphi_{i,k},\varphi_{i,l}\rangle=\int_0^\tau \varphi_{i,k}(t)\,\varphi_{i,l}(t)\,dt=\delta_{kl}$.

The FD model $X_d(t)$ of (\ref{201b}) has two notable properties.  First, the random vector  $X_d(t)$ converges in m.s. to $X(t)$ for any $t$, i.e., $E[||X_d(t)-X(t)||_2^2]$ as $\min_{1\le i \le n}d_i\to\infty$, where $||\cdot||_2$ denotes the $L_2$ norm. This results from the equalities
\begin{eqnarray}\label{mean-error}
E[ \ ||X_d(t)-X(t)||_2^2 \ ]&=&\sum_{i=1}^{n}E[ \ (X_{i,d_i}(t)-X_i(t))^2 \ ]
=\sum_{i=1}^{n}\sum_{k=d_i+1}^{\infty}E[Z_{i,k}^2]\varphi_{i,k}(t)^2\nonumber\\
&=&\sum_{i=1}^{n}\sum_{k=d_i+1}^{\infty}\lambda_{i,k}\varphi_{i,k}(t)^2
\end{eqnarray}
and the convergence $\sum_{k=d_i+1}^{\infty}\lambda_{i,k}\varphi_{i,k}(t)^2\to 0$ as $d_i\to 0$, $i=1,\ldots,n$, implied by the Mercer Theorem \cite{J1909}.

The second property is that the finite dimensional distributions of $X_d(t)$ converge to those of $X(t)$ as $\min_{1\le i \le n}d_i\to\infty$.  Let $l\geq 1$ be an arbitrary integer and $(t_1,\ldots,t_l)$ be $l$ arbitrary times in $[0,\tau]$.  Denote by $F_{\mathcal{X}}$ and $F_{\mathcal{X}_d}$ the distributions of the $nl$-dimensional vectors
$\mathcal{X}=(X(t_1)^T,\ldots,X(t_l)^T)^T$ and $\mathcal{X}_d=(X_d(t_1)^T,\ldots,X_d(t_l)^T)^T$, i.e., the finite dimensional distributions of $X(t)$ and $X_d(t)$. Since $X_{i,d_i}(t_j)$ converges in m.s. to $X_i(t_j)$ as $d_i\to\infty$ for any $i=1,\cdots,n$ and $j=1,\cdots,l$, then $\mathcal{X}_d$ converges in m.s. to $\mathcal{X}$ in $L_{2}$ norm as $\min_{1\le i \le n}d_i\to\infty$. Then, $\mathcal{X}_d$ converges to $\mathcal{X}$ in probability since for any $\varepsilon>0$, $P(||\mathcal{X}_d-\mathcal{X}||_2>\varepsilon)\leq E[ \ ||\mathcal{X}_d-\mathcal{X}||_2 \ ]/\varepsilon$ by Chebyshev's inequality \cite{F1982}.  This implies the convergence  $F_{\mathcal{X}_d}\to F_{\mathcal{X}}$ as $\min_{1\le i \le n}d_i\to\infty$, i.e., the convergence of the finite dimensional distributions of $X_d(t)$ to those of $X(t)$, see Theorem 18.10 in \cite{v1998}.

\section{Main results}
\setcounter{thm}{0}\setcounter{Corol}{0}\setcounter{lemma}{0}\setcounter{pron}{0}\setcounter{equation}{0}
\setcounter{remark}{0}\setcounter{exam}{0}\setcounter{property}{0}\setcounter{defin}{0}

Let  $X(t)$ and $X_d(t)$, $0\leq t\leq \tau$, be $\mathbb{R}^n$-valued zero-mean processes defined on the same probability space $\big(\Omega,{\cal F},P\big)$. It is assumed that $X(t)$ has continuous samples and continuous correlation functions $c_{ij}(s,t)$ on $[0,\tau]^2$, 
so that the samples of $X(t)$ and $X_d(t)$ are elements of the space of $n$-dimensional continuous function $C[0,\tau]$.  Note also that the samples of these two processes are paired by construction since the samples of $\{Z_{i,k}\}$ are obtained from samples of $X_i(t)$ by projection, see (\ref{202}). Our objective is to show that $X_d(t)$ converges weakly to $X(t)$ in the metric of $C[0,\tau]$, a convergence which is denoted by $X_d \Rightarrow X$. According to Theorem 8.1 in \cite{billingsley68}, the family of processes $\{X_d(t)\}$ converges weakly to $X(t)$ in $C[0,\tau]$ if
(1) the finite dimensional distributions $F_{\mathcal{X}_d}$ of $X_d(t)$ converge to $F_{\mathcal{X}}$, a convergence already established, and (2) the family of processes $\{X_d(t)\}$ is tight in  $C[0,\tau]$. We use the tightness criterion of Theorem 8.2 in \cite{billingsley68}. The following theorem is our main result.

\begin{thm}\label{thm2}
If $X(t)$ has continuous samples, the correlation functions of the components of $X(t)$ are continuous, $F_{\mathcal{X}_d}\to F_{\mathcal{X}}$ as $\min_{1\le i \le n}d_i\to\infty$ and
$\sum_{k=1}^{\infty}\lambda_{i,k}C_{i,k}<\infty$ for any $i=1,\cdots,n$, then
\begin{eqnarray}\label{thm2-2}
\sup_{t\in[0,\tau]}||X_d(t)||\overset{w}{\to}\sup_{t\in[0,\tau]}||X(t)||, \ \min_{1\le i \le n} d_i\to\infty,
\end{eqnarray}
and
\begin{eqnarray}\label{thm2-1}
\sup_{t\in[0,\tau]}|X_{i,d_i}(t)|\overset{w}{\to}\sup_{t\in[0,\tau]}|X_i(t)|, \ d_i\to\infty, \ i=1,\cdots,n
\end{eqnarray}
where $w$ denotes weak convergence, $X_d(t)$ is given by $(\ref{201b})$, $C_{i,k}=\sup_{t\in[0,\tau]}\varphi_{i,k}(t)^2$ and $||\cdot||$ denotes the $L_{\infty}$ norm.
\end{thm}
\begin{proof}

We first show that the family of random vectors $\{X_d(0)\}$ is tight, i.e., for arbitrary $\varepsilon>0$, there exists $L>0$ such that $P(||X_d(0)||>L)\leq\varepsilon$, see Theorem 8.2 in \cite{billingsley68}.  Note that
\begin{eqnarray*}
E[X_{i,d_i}(0)^2]=\sum_{k=1}^{d}\lambda_{i,k}\varphi_{i,k}(0)^2\leq\sum_{k=1}^{\infty}\lambda_{i,k}\varphi_{i,k}(0)^2=E[X_i(0)^2]<\infty, \ \forall \ d\ge1, \ i=1,\cdots,n,
\end{eqnarray*}
since $X_i(t)$ has finite variance, so that
\begin{eqnarray*}
P(||X_d(0)||>L)\leq \frac{1}{L}E[ \ ||X_d(0)|| \ ]\leq\frac{1}{L}\sum_{i=1}^{n}E[ \ |X_{i,d_i}(0)| \ ]\leq
\frac{1}{L}\sum_{i=1}^{n}\Big(E[X_{i,d_i}(0)^2]\Big)^{1/2}, \forall d\ge1    
\end{eqnarray*}
by Chebyshev's inequality. Since $\sum_{i=1}^{n}\Big(E[X_{i,d_i}(0)^2]\Big)^{1/2}$ is finite, there exits $L>0$ such that $P(||X_d(0)||>L)\leq\varepsilon$ for any $\varepsilon>0$, so that $\{X_d(0)\}$ is tight.

We consider now the second condition of Theorem~8.2, which requires to show that, for given $\varepsilon,\eta>0$, there exists $\delta>0$ such that $P\big(W(\delta)\geq \varepsilon\big)\leq \eta$,  
where $W(\delta)=\sup_{|s-t|\leq\delta}||X_d(t)-X_d(s)||$ denotes the modulus of continuity of $X_d(t)$ and $||\cdot|| $ is the $L_{\infty}$ norm. Note that
\begin{eqnarray*}
P\big(W(\delta)\geq \varepsilon\big)&=&P\bigg(\sup_{|s-t|\leq\delta}||X_d(t)-X_d(s)||\geq\varepsilon\bigg)\nonumber\\
&\leq&P\bigg(\bigcup_{i=1}^{n}\sup_{|s-t|\leq\delta}|X_{i,d_i}(t)-X_{i,d_i}(s)|\geq\varepsilon\bigg)\nonumber\\
&\leq&\sum_{i=1}^{n}P\bigg(\sup_{|s-t|\leq\delta}|X_{i,d_i}(t)-X_{i,d_i}(s)|\geq\varepsilon\bigg)\nonumber\\
&\leq&\frac{1}{\varepsilon^2}\sum_{i=1}^{n}E\bigg[\bigg(\sup_{|s-t|\leq\delta}|X_{i,d_i}(t)-X_{i,d_i}(s)|\bigg)^2\bigg]
\end{eqnarray*}
by Chebyshev's inequality. Since $\sup_{s,t\in[0,\tau]}f(s,t)^2=(\sup_{s,t\in[0,\tau]}|f(s,t)|)^2$ and $\sup_{s,t\in[0,\tau]}(f(s,t)+g(s,t))\leq\sup_{s,t\in[0,\tau]}f(s,t)+\sup_{s,t\in[0,\tau]}g(s,t)$, then we have
\begin{eqnarray*}
&&E\bigg[\bigg(\sup_{|s-t|\leq\delta}|X_{i,d_i}(t)-X_{i,d_i}(s)| \bigg)^2\bigg]
=E\bigg[\sup_{|s-t|\leq\delta}\Big(X_{i,d_i}(t)-X_{i,d_i}(s)\Big)^2\bigg]\nonumber\\
&=&E\bigg[\sup_{|s-t|\leq\delta}\Big(\sum_{k=1}^{d_i}Z_{i,k}\big(\varphi_{i,k}(t)-\varphi_{i,k}(s)\big)\Big)^2\bigg]\nonumber\\
&=&E\bigg[\sup_{|s-t|\leq\delta}\Big(\sum_{k=1}^{d_i}Z_{i,k}^2\big(\varphi_{i,k}(t)-\varphi_{i,k}(s)\big)^2
+\sum_{1\le k\not=l\le d_i}Z_{i,k}Z_{i,l}\big(\varphi_{i,k}(t)-\varphi_{i,k}(s)\big)\big(\varphi_{i,l}(t)-\varphi_{i,l}(s)\big)\Big)\bigg]\nonumber\\
&\leq&E\bigg[\sum_{k=1}^{d_i}Z_{i,k}^2\sup_{|s-t|\leq\delta}\big(\varphi_{i,k}(t)-\varphi_{i,k}(s)\big)^2\nonumber\\
& &+\sum_{1\le k\not=l\le d_i}Z_{i,k}Z_{i,l}\sup_{|s-t|\leq\delta}\big(\varphi_{i,k}(t)-\varphi_{i,k}(s)\big)\big(\varphi_{i,l}(t)-\varphi_{i,l}(s)\big)\bigg]\nonumber\\
&=&\sum_{k=1}^{d_i}E[Z_{i,k}^2]\sup_{|s-t|\leq\delta}\big(\varphi_{i,k}(t)-\varphi_{i,k}(s)\big)^2
\leq\sum_{k=1}^{\infty}\lambda_{i,k}\sup_{|s-t|\leq\delta}\big(\varphi_{i,k}(t)-\varphi_{i,k}(s)\big)^2.
\end{eqnarray*}
Let $L_{i,k}(\delta)=\lambda_{i,k}\sup_{|s-t|\leq\delta}\big(\varphi_{i,k}(t)-\varphi_{i,k}(s)\big)^2$, $\delta\in[0,\tau]$, and note that
\begin{eqnarray*}
\sum_{k=1}^{\infty}L_{i,k}(\delta)\leq2\sum_{k=1}^{\infty}\lambda_{i,k}\sup_{s,t\in[0,\tau]}\big(\varphi_{i,k}(t)^2+\varphi_{i,k}(s)^2\big)
\leq4\sum_{k=1}^{\infty}\lambda_{i,k}C_{i,k}<\infty.
\end{eqnarray*}
Since $L_{i,k}(\delta)\leq 4\lambda_{i,k}C_{i,k}$, for $k=1,2,\ldots$ and $i=1,\ldots,n$, and the series $\sum_{k=1}^{\infty}\lambda_{i,k}C_{i,k}$ is convergent, $\sum_{k=1}^\infty L_{i,k}(\delta)$ converges uniformly on $[0,\tau]$ by Theorem 7.10 in \cite{R1976}.  Since $L_{i,k}(\delta)\to 0$ as $\delta\to0$ by the continuity of the eigenfunctions $\{\varphi_{i,k}\}$, we have

\begin{eqnarray*}
\lim_{\delta\to0}\sum_{k=1}^{\infty}L_{i,k}(\delta)=\sum_{k=1}^{\infty}\lim_{\delta\to0}L_{i,k}(\delta)=0, \quad i=1,\cdots,n
\end{eqnarray*}
by Theorem 7.11 in \cite{R1976}. Therefore, for given $\varepsilon,\eta>0$, there exists $\delta$ such that
\begin{eqnarray*}
P\big(W(\delta)\geq \varepsilon\big)\leq \frac{1}{\varepsilon^2}\sum_{i=1}^{n}\sum_{k=1}^{\infty}\lambda_{i,k}\sup_{|s-t|\leq\delta}\big(\varphi_{i,k}(t)-\varphi_{i,k}(s)\big)^2<\eta,
\end{eqnarray*}
which means that $X_d(t)$ is tight in $C[0,\tau]$.

If $X(t)$ is a real-valued stochastic process, i.e., $n=1$, the conditions of the above theorem reduce to $\sum_{k=1}^{\infty}\lambda_{k}C_{k}<\infty$, where $C_{k}=\sup_{t\in[0,\tau]}\varphi_{k}(t)^2$, $\lambda_k$ and $\varphi_k(t)$ are the eigenvalue and eigenfunction of the correlation function of $X(t)$, and  (\ref{thm2-2}) holds under these conditions. This implies that (\ref{thm2-1}) holds for each real-valued processes $X_{i}(t)$.
\end{proof}

The practical implication of the weak convergence in (\ref{thm2-2}) and (\ref{thm2-1}) is that the distributions of $\sup_{t\in[0,\tau]}||X(t)||$ and $\sup_{t\in[0,\tau]}|X_i(t)|$ can be estimated from FD samples for sufficiently large truncation levels $\{d_i\}$.

Generally, the conditions of the above theorem are difficult to check, since eigenvalues and eigenfunctions are not available analytically. An example in which these conditions can be checked is that of a real-valued, zero-mean weakly stationary process $X(t)$,  $t\in[-\tau,\tau]$, with correlation function $c(s,t)=(1+e^{-2\gamma|s-t|})/4$, where $\gamma>1/(2\tau)$ and $s,t\in[-\tau,\tau]$. We define $X_d(t)$ as (\ref{201b}). The eigenfunction of $c(s,t)$ are (see Example 6-4.1 in \cite{dr1987})
\begin{eqnarray*}
\varphi_{2k-1}(t)=\frac{\cos(2\gamma a_k t)}{\sqrt{\tau+\sin(4\gamma\tau a_k)/(4\gamma a_k)}} \ \ \text{and} \ \
\varphi_{2k}(t)=\frac{\sin(2\gamma b_k t)}{\sqrt{\tau-\sin(4\gamma\tau b_k)/(4\gamma b_k)}}, \ \ k=1,2,\cdots,
\end{eqnarray*}
where $a_k$ and $b_k$ are solutions of $a_k\tan(2\gamma \tau a_k)=1$ and $b_k\cot(2\gamma \tau b_k)=1$. Note that
\begin{eqnarray*}
|\varphi_{2k-1}(t)|=\frac{|\cos(2\gamma a_k t)|}{\sqrt{\tau+\sin^2(2\gamma\tau a_k)/(2\gamma)}}\leq \frac{1}{\sqrt{\tau}},  \ \ k=1,2,\cdots
\end{eqnarray*}
and
\begin{eqnarray*}
|\varphi_{2k}(t)|=\frac{|\sin(2\gamma b_k t)|}{\sqrt{\tau-\cos^2(2\gamma\tau b_k)/(2\gamma)}}\leq \frac{1}{\sqrt{\tau-1/(2\gamma)}}, \ \ k=1,2,\cdots,
\end{eqnarray*}
which implies $|\varphi_k(t)|\leq 1/\sqrt{\tau-1/(2\gamma)}$ for all $k\ge1$. Then
\begin{eqnarray*}
\sum_{k=1}^{\infty}\lambda_{k}C_{k}\leq\frac{1}{\sqrt{\tau-1/(2\gamma)}}\sum_{k=1}^{\infty}\lambda_{k}<\infty,
\end{eqnarray*}
since $\sum_{k=1}^{\infty}\lambda_{k}<\infty$ \cite[Chap.~4]{GG1980}. Therefore, from Theorem $\ref{thm2}$, we have $\sup_{t\in[0,\tau]}|X_d(t)|\overset{w}{\to}\sup_{t\in[0,\tau]}|X(t)|$ as $d\to\infty$.

The following two theorems give alternative conditions for the weak convergence of FD models to target processes.  These theorems are useful in applications since their conditions
are simpler to check and are satisfied by a broad range of processes.

\begin{thm}\label{thm3}
If the second derivative of the correlation functions $\{c_{ii}(s,t)\}$ of the components of $X(t)$ exist and are continuous in $[0,\tau]^2$, then
\begin{eqnarray*}
\sup_{t\in[0,\tau]}||X_d(t)-X(t)||\overset{p}{\to}0, \ \min_{1\le i \le n} d_i\to\infty,
\end{eqnarray*}
where $p$ denotes convergence in probability.
\end{thm}

\begin{proof}
For any $\varepsilon>0$,
\begin{eqnarray*}
&&P\bigg(\sup_{t\in[0,\tau]}||X_d(t)-X(t)||>\varepsilon\bigg)\leq\frac{1}{\varepsilon}E\bigg[\sup_{t\in[0,\tau]}||X_d(t)-X(t)||\bigg]\nonumber\\
&\leq&\frac{1}{\varepsilon}E\bigg[\sup_{t\in[0,\tau]}\sum_{i=1}^{n}|X_{i,d_i}(t)-X_i(t)|\bigg]
\leq\frac{1}{\varepsilon}\sum_{i=1}^{n}E\bigg[\sup_{t\in[0,\tau]}|X_{i,d_i}(t)-X_i(t)|\bigg].
\end{eqnarray*}
by the relationship $\sup_{t\in[0,\tau]}(f(t)+g(t))\leq\sup_{t\in[0,\tau]}f(t)+\sup_{t\in[0,\tau]}g(t)$ and Chebyshev's inequality. Since $\dot{X}(t)$ exists in the mean square sense, we have
\begin{eqnarray*}
&&E\bigg[\sup_{t\in[0,\tau]}|X_{i,d_i}(t)-X_i(t)|\bigg]\nonumber\\
&=&E\bigg[\sup_{t\in[0,\tau]}\bigg|\int_{0}^{t}\dot{X}_{i,d_i}(s)ds-\int_{0}^{t}\dot{X}_i(s)ds+X_{i,d_i}(0)-X_i(0)\bigg|\bigg]\nonumber\\
&\leq&E\bigg[\sup_{t\in[0,\tau]}\int_{0}^{t}|\dot{X}_{i,d_i}(s)-\dot{X}_i(s)|ds\bigg]+E[ \ |X_{i,d_i}(0)-X_i(0)| \ ]\nonumber\\
&=&E\bigg[\int_{0}^{\tau}|\dot{X}_{i,d_i}(s)-\dot{X}_i(s)|ds\bigg]+E[ \ |X_{i,d_i}(0)-X_i(0)| \ ]\nonumber\\
&=&\int_{0}^{\tau}E[ \ |\dot{X}_{i,d_i}(s)-\dot{X}_i(s)| \ ]ds+E[ \ |X_{i,d_i}(0)-X_i(0)| \ ]\nonumber\\
&\leq&\int_{0}^{\tau}\Big(E[ \ (\dot{X}_{i,d_i}(s)-\dot{X}_i(s))^2 \ ]\Big)^{1/2}ds+\Big(E[ \ (X_{i,d_i}(0)-X_i(0))^2 \ ]\Big)^{1/2}, \ i=1,\cdots,n,
\end{eqnarray*}
where the change of order of integration holds by Fubini's theorem. Note that
\begin{eqnarray*}
X_{i,d_i}(s)=\sum_{k=1}^{d_i}\lambda_{i,k}\varphi_{i,k}(s)^2 \ \ \text{and} \ \ \dot{X}_{i,d_i}(s)=\sum_{k=1}^{d_i}\lambda_{i,k}\dot{\varphi}_{i,k}(s)^2, \ i=1,\cdots,n, \ s\in[0,\tau].
\end{eqnarray*}
Since the second derivative of $c_{ii}(s,t)$ is continuous for each $i=1,\cdots,N$, the eigenfunctions $\{\varphi_{i,k}(s)\}$ and $\{\dot{\varphi}_{i,k}(s)\}$ are continuous \cite{K1967} so that calculations as in (\ref{mean-error}) give
\begin{eqnarray*}
E[ \ (X_{i,d_i}(s)-X_i(s))^2 \ ]=\sum_{k=d_i+1}^{\infty}\lambda_{i,k}\varphi_{i,k}(s)^2\to0, \ d_i\to\infty
\end{eqnarray*}
and
\begin{eqnarray*}
E[ \ (\dot{X}_{i,d_i}(s)-\dot{X}_i(s))^2 \ ]=\sum_{k=d_i+1}^{\infty}\lambda_{i,k}\dot{\varphi}_{i,k}(s)^2\to0, \ d_i\to\infty,
\end{eqnarray*}
where the convergence is uniform in $s\in[0,\tau]$ for each $i=1,\cdots,n$ by Mercer's theorem.  This implies
\begin{eqnarray}\label{thm2-101}
E\bigg[\sup_{t\in[0,\tau]}|X_{i,d_i}(t)-X_i(t)|\bigg]\to0, \ d_i\to\infty, \ i=1,\cdots,n.
\end{eqnarray}
so that for any $\varepsilon>0$,
\begin{eqnarray*}
P\bigg(\sup_{t\in[0,\tau]}||X_d(t)-X(t)||>\varepsilon\bigg)&=&P\bigg(\bigcup_{i=1}^{n}\sup_{t\in[0,\tau]}|X_{i,d_i}(t)-X_i(t)|>\varepsilon\bigg)\nonumber\\
&\leq&\sum_{i=1}^{n}P\bigg(\sup_{t\in[0,\tau]}|X_{i,d_i}(t)-X_i(t)|>\varepsilon\bigg)\nonumber\\
&\leq&\frac{1}{\varepsilon}\sum_{i=1}^{n}E\bigg[\sup_{t\in[0,\tau]}|X_{i,d_i}(t)-X_i(t)|\bigg]\to0, \ \min_{1\le i\le n}d_i\to\infty
\end{eqnarray*}
by Chebyshev's inequality and (\ref{thm2-101}). Hence,
\begin{eqnarray*}
\sup_{t\in[0,\tau]}||X_d(t)-X(t)||\overset{p}{\to}0, \ \min_{1\le i \le n} d_i\to\infty.
\end{eqnarray*}
\end{proof}
%

\begin{thm}\label{thm4-1}
Let $G(t)=(G_1(t),\cdots,G_n(t))^T$ be a zero-mean Gaussian process with continuous samples and continuous correlation function and let $\{G_d(t)\}$ be the FD models of $G(t)$ defined by (\ref{201b}).  Then
\begin{eqnarray}\label{thm4-eq-1}
\sup_{t\in[0,\tau]}||G_d(t)-G(t)||\overset{p}{\to}0, \ \min_{1\le i \le n}d_i\to\infty,
\end{eqnarray}
where $p$ denotes convergence in probability.
\end{thm}

\begin{proof}
Note that $G_{i,d_i}(t)-G_i(t)$ is a zero-mean Gaussian process defined on the closed interval $[0,\tau]$ for any $i=1,2,\cdots,n$, so that there exists constant $K_i$ such that
\begin{eqnarray*}
P\bigg(\sup_{t\in[0,\tau]}\Big(G_{i,d_i}(t)-G_i(t)\Big)>\varepsilon\bigg)\leq K_ie^{-\varepsilon^2/4\sigma_{i,d_i}^2},\quad \varepsilon>0
\end{eqnarray*}
by \cite{ms1971} lemma 3.1 and \cite{gs1991}, where
\begin{eqnarray*}
\sigma_{i,d_i}^2&=&\sup_{t\in[0,\tau]}{\rm Var}[G_{i,d_i}(t)-G_i(t)]=\sup_{t\in[0,\tau]}E\Big[\Big(G_{i,d_i}(t)-G_i(t)\Big)^2\Big]\nonumber\\
&=&\sup_{t\in[0,\tau]}\sum_{k=d_i+1}^{\infty}\lambda_{i,k}\varphi_{i,k}(t)^2\to0, \ d_i\to\infty
\end{eqnarray*}
by Mercer's theorem. Then, for any $\varepsilon>0$, we have
\begin{eqnarray*}
&&P\bigg(\sup_{t\in[0,\tau]}|G_{i,d_i}(t)-G_i(t)|>\varepsilon\bigg)\nonumber\\
&=&P\bigg(\bigg\{\sup_{t\in[0,\tau]}\Big(G_{i,d_i}(t)-G_i(t)\Big)>\varepsilon\bigg\}\bigcup
\bigg\{\sup_{t\in[0,\tau]}\Big(G_i(t)-G_{i,d_i}(t)\Big)>\varepsilon\bigg\}\bigg)\nonumber\\
&\leq&P\bigg(\sup_{t\in[0,\tau]}\Big(G_{i,d_i}(t)-G_i(t)\Big)>\varepsilon\bigg)+
P\bigg(\sup_{t\in[0,\tau]}\Big(G_i(t)-G_{i,d_i}(t)\Big)>\varepsilon\bigg)\nonumber\\
&\leq&2K_ie^{-\varepsilon^2/4\sigma_{i,d_i}^2}\to0, \ d_i\to\infty, \ i=1,\cdots,n,
\end{eqnarray*}
which means that $\sup_{t\in[0,\tau]}|G_{i,d_i}(t)-G_i(t)|$ converges in probability to zero, as $d_i\to\infty$, $i=1,\cdots,n$. Therefore
\begin{eqnarray*}
\sup_{t\in[0,\tau]}||G_d(t)-G(t)||\leq \sum_{i=1}^{n}\sup_{t\in[0,\tau]}|G_{i,d_i}(t)-G_i(t)|\overset{p}{\to}0, \ \min_{1\le i \le n}d_i\to\infty.
\end{eqnarray*}
\end{proof}

The latter result extends directly to the class of non-Gaussian translation processes $X(t)$. The components of these processes are defined by
\begin{eqnarray}\label{t-1}
X_i(t)=F_i^{-1}\circ\Phi(G_i(t)),\quad i=1,\ldots,n,
\end{eqnarray}
where $\big(G_1(t),\ldots,G_n(t)\big)$ is a stationary vector-valued Gaussian process whose components have zero means and unit variances, $\Phi$ denotes the distribution of the standard normal variable and $F_i$ is the marginal distribution of $X_i(t)$. The FD models $\{X_d(t)\}$ of $X(t)$ have the form
\begin{eqnarray}\label{t-2}
X_{i,d_i}(t)=F_i^{-1}\circ \Phi(G_{i,d_i}(t)),\quad i=1,\ldots,n,
\end{eqnarray}
where $G_{i,d_i}(t)$ is a finite dimensional model of a Gaussian process $G_i(t)$, see (\ref{201b}).

\begin{Corol}\label{cor1}
Let $X(t)$, $X_d(t)$ be defined in $(\ref{t-1})$ and $(\ref{t-2})$. If $G_i(t)$ in $(\ref{t-1})$ satisfies the conditions of Theorem \ref{thm4-1} and $F_i$ is continuous and strictly monotonically increasing for each $i=1,\cdots,n$, then

\begin{eqnarray*}
\sup_{t\in[0,\tau]}||X_d(t)-X(t)||\overset{p}{\to}0, \ \min_{1\le i \le n}d_i\to\infty.
\end{eqnarray*}
\end{Corol}

\begin{proof}


Let $V_d(t)=(V_{1,d_1}(t),\cdots,V_{n,d_n}(t))^T$ and $V(t)=(V_1(t),\cdots,V_n(t))^T$, where $V_{i,d_i}(t)=\Phi(G_{i,d_i}(t))$ and $V_{i}(t)=\Phi(G_{i}(t))$, $i=1,\cdots,n$. Then according to Theorem \ref{thm4-1} and mean value theorem,
\begin{eqnarray*}
\sup_{t\in[0,\tau]}|V_{i,d_i}(t)-V_i(t)|\leq\frac{1}{\sqrt{2\pi}}\sup_{t\in[0,\tau]}|G_{i,d_i}(t)-G_i(t)|\overset{p}{\to}0, \ d_i\to\infty, \ i=1,\cdots,n.
\end{eqnarray*}
Since $F_i^{-1}$ is continuous, then $F_i^{-1}$ is uniformly continuous on $[0,1]$ for each $i=1,\cdots,n$, which leads to
\begin{eqnarray*}
\sup_{t\in[0,\tau]}||X_d(t)-X(t)||&\leq&\sum_{i=1}^{n}\sup_{t\in[0,\tau]}|X_{i,d_i}(t)-X(t)|\nonumber\\
&=&\sum_{i=1}^{n}\sup_{t\in[0,\tau]}|F_i^{-1}(V_{i,d_i}(t))-F_i^{-1}(V_i(t))|\overset{p}{\to}0, \ \min_{1\le i\le n}d_i\to\infty.
\end{eqnarray*}
\end{proof}

We also consider the FD model $X^{(N)}(t)=(X^{(N)}_1(t),\cdots,X^{(N)}_n(t))^T$ whose samples interpolate linearly between values of $X(t)$ at the times $(t_0,t_1,\ldots,t_N)$, where $t_i=i\Delta t$, $0\le i\le N$ and $\Delta t=\tau/N$. The following theorem shows that, under some proper conditions, the discrepancy between the samples of $X^{(N)}(t)$ and $X(t)$ measured by the metric of $C[0,\tau]$ can be made as small as desired by increasing $N$.

\begin{thm}\label{thm5-2}
Let $X(t)$ be a $n$-dimensional vector valued process on $[0,\tau]$. If there exists $\kappa_1,\kappa_2,C>0$ such that
\begin{eqnarray}\label{thm5-5}
E[ \ ||X(t+\delta)-X(t)||^{\kappa_1} \ ]\leq C\delta^{1+\kappa_2}, \quad t,t+\delta\in[0,\tau],
\end{eqnarray}
then
\begin{eqnarray}\label{thm5-3}
\sup_{t\in[0,\tau]}||X^{(N)}(t)-X(t)||\overset{a.s.}{\to}0, \ N\to\infty,
\end{eqnarray}
where a.s.~denotes almost sure convergence.
\end{thm}

\begin{proof}
For any $t\in[t_{i-1},t_i]$, $i=1,\cdots,N$, let $t=t_{i-1}+\xi$, $\xi\in[0,\Delta t]$, we have
\begin{eqnarray*}
X^{(N)}(t)-X(t)
&=&X(t_{i-1})+\frac{1}{\Delta t} \Big(X(t_i)-X(t_{i-1})\Big)(t-t_{i-1})-X(t)\nonumber\\
&=&\Big(1-\frac{\xi}{\Delta t}\Big)\Big(X(t_{i-1})-X(t_{i-1}+\xi)\Big)+\frac{\xi}{\Delta t} \Big(X(t_i)-X(t_{i-1}+\xi)\Big) \end{eqnarray*}
for almost all samples so that
\begin{eqnarray*}
&&\sup_{t\in[0,\tau]}||X^{(N)}(t)-X(t)||
=\max_{1\le i\le N}\sup_{t\in[t_{i-1},t_i]}||X^{(N)}(t)-X(t)||\nonumber\\
&=&\max_{1\le i\le N}\sup_{\xi\in[0,\Delta t]}\bigg|\bigg|\Big(1-\frac{\xi}{\Delta t}\Big)\Big(X(t_{i-1})-X(t_{i-1}+\xi)\Big)+\frac{\xi}{\Delta t} \Big(X(t_i)-X(t_{i-1}+\xi)\Big)\bigg|\bigg|\nonumber\\
&\leq&\max_{1\le i\le N}\sup_{\xi\in[0,\Delta t]}\Big(1-\frac{\xi}{\Delta t}\Big)\bigg|\bigg|X(t_{i-1})-X(t_{i-1}+\xi)\bigg|\bigg|
+\max_{1\le i\le N}\sup_{\xi\in[0,\Delta t]}\frac{\xi}{\Delta t} \bigg|\bigg|X(t_i)-X(t_{i-1}+\xi)\bigg|\bigg|\nonumber\\
&\leq&2\sup_{|s-t|\leq\Delta t}||X(s)-X(t)||=2\sup_{|s-t|\leq\tau/N}||X(s)-X(t)||.
\end{eqnarray*}

Under  the condition (\ref{thm5-5}), which constitutes the Kolmogorov continuity criterion (\cite{WH1985}, Proposition 4.2), we have $\sup_{|s-t|\leq\tau/N}||X(s)-X(t)||\to0$ almost surely, as $N\to\infty$ so that $\sup_{t\in[0,\tau]}||X^{(N)}(t)-X(t)||\to0$ almost surely, as $N\to\infty$.
%
%
\end{proof}



\section{Numerical illustrations}
\setcounter{thm}{0}\setcounter{Corol}{0}\setcounter{lemma}{0}\setcounter{pron}{0}\setcounter{equation}{0}
\setcounter{remark}{0}\setcounter{exam}{0}\setcounter{property}{0}\setcounter{defin}{0}

The previous section provides conditions under which the distribution of extremes of FD models converge to that of extremes of target processes (Theorem~\ref{thm2}), the probability of the discrepancy between FD and target processes can be made as small as desired  (Theorems~\ref{thm3} and \ref{thm4-1}) and the discrepancy between FD and target samples can be eliminated (Theorem~\ref{thm5-2}).  The conditions of Theorem~\ref{thm2} are difficult to validate and, if satisfied, they guarantee that the distribution of extremes of target functions can be approximated by that of FD models for a sufficiently large stochastic dimension.  If the conditions of  Theorems~\ref{thm3} and \ref{thm4-1} hold, then the subset
$$
\Omega_d(\varepsilon)=\{\omega: \sup_{t\in[0,\tau]}||X_d(t,\omega)-X(t,\omega)||>\varepsilon\}
$$
in which the discrepancy between FD and target samples in the metric of $C[0,\tau]$ exceeds an arbitrary $\varepsilon>0$ can be made as small as desired by increasing $\min_{1\le i \le n}d_i$ since $P\big(\Omega_d(\varepsilon)\big)\to 0$ as $\min_{1\leq i\leq n}d_i\to 0$. Accordingly, most of FD samples provide accurate representations of the corresponding target samples.


\begin{exam}\label{exam1}
{\rm Let $X_1(t)=Y_1(t)$ and $X_2(t)=Y_1(t)+Y_2(t)$, $t\in[0,\tau]$, where $Y_i(t)=F_i^{-1}\circ\Phi(G_i(t))$, follows the Gumbel distribution $F_i(x)=\exp\{-\exp\{(x-\mu_i)/\gamma_i\}\}$ with parameters $\mu_i\in\mathbb{R}$, $\gamma_i>0$ and $\{G_i(t)\}$ are zero-mean independent Gaussian processes with correlation functions $c_i(s,t)=(1+\nu_i|s-t|)e^{-\nu_i|s-t|}$, $\nu_i>0$, $i=1,2$.  The following numerical results are for $\mu_1=0$, $\mu_2=1$, $\gamma_1=1$, $\gamma_2=2$, $\nu_1=0.1$, $\nu_2=0.2$ and $\tau=50$.  All reported statistics are based on 5000 samples.  The reference samples are those of $X^{(N)}(t)$ in Theorem~3.4 with $\Delta t=\tau/N=0.01$ and $N=5000$.  }
\end{exam}

We construct two types of FD models for $X(t)=(X_1(t),X_2(t))^T$. The first FD model has the form
\begin{eqnarray*}
X_{1,d_1}(t)=Y_{1,d_1}(t) \ \text{and} \ X_{2,d_2}(t)=Y_{1,d_2}(t)+Y_{2,d_2}(t),
\end{eqnarray*}
where $Y_{1,d_1}(t)$, $Y_{1,d_2}(t)$ and $Y_{2,d_2}(t)$ are FD models of $Y_1(t)$ and $Y_2(t)$ defined by (\ref{201b}) with $n=1$. Since the distributions $F^{-1}$ and $\Phi$ and the correlation functions of $Y_1(t)$ and $Y_2(t)$ are differentiable, $\sup_{t\in[0,\tau]}|Y_{i,d_i}(t)-Y_i(t)|$ converges weakly to $0$ for $i=1,2$ as $\min\{d_1,d_2\}\to\infty$ by Theorem \ref{thm3}, which implies the weak convergence  $\sup_{t\in[0,\tau]}|X_{i,d_i}(t)-X_i(t)|\to 0$, $i=1,2$.  This means that most of the FD samples will approximate accurately the corresponding target samples for a sufficiently large stochastic dimension.

The left, middle and right panels of Figs.~\ref{ex1-fig1} and \ref{ex1-fig6} show with solid and dotted lines five samples of $X_i(t)$ and $X_{i,d_i}(t)$, $i=1,2$ for $d_1=5,10,15$ and $d_2=15,20,25$. Scatter plots of $\big(\sup_{0\leq t\leq \tau}|X_i(t)|,\sup_{0\leq t\leq \tau}|X_{i,d_i}(t)|\big)$ are in Figs.~\ref{ex1-fig4} and \ref{ex1-fig5} for the same values of $d_i,i=1,2$ (left, middle and right panels).

The thin solid lines of the left and right panels of Fig.~\ref{ex1-fig7} are estimates of $P(\sup_{t\in[0,\tau]}|X_i(t)|$ \\ $>x)$ for $i=1$ and $i=2$ which are obtained directly from data. These probabilities are viewed as reference. The other lines of the figure are calculated from FD models with $d_1=5$ and $d_2=15$ (heavy solid lines), $d_1=10$ and $d_2=20$ (dotted lines) and $d_1=15$ and $d_2=25$ (dashed lines) for the first and second components (left and right panels). The dashed lines are the closest to the reference.
These plots show, in agreement with our theoretical results, that the discrepancy between samples and extremes of $X_i(t)$ and $X_{i,d_i}(t)$ can be made as small as desired by increasing the stochastic dimension $d_i$, $i=1,2$. The FD models of $X_1(t)$ require a lower truncation level since the samples of $X_1(t)$ are smoother than those of $X_2(t)$, see Figs.~\ref{ex1-fig1} and \ref{ex1-fig6}


\begin{figure}[H]
  \centering
  \includegraphics[scale=0.35]{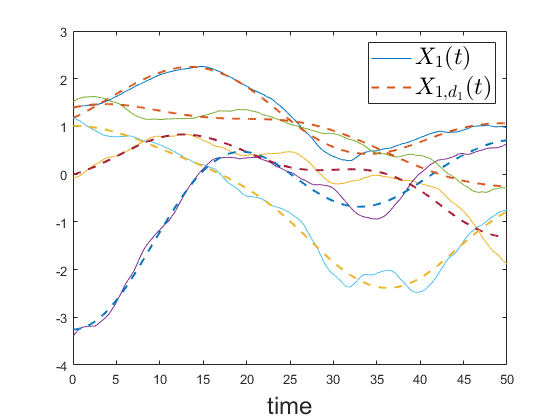}
  \hspace{0.1in}
  \includegraphics[scale=0.35]{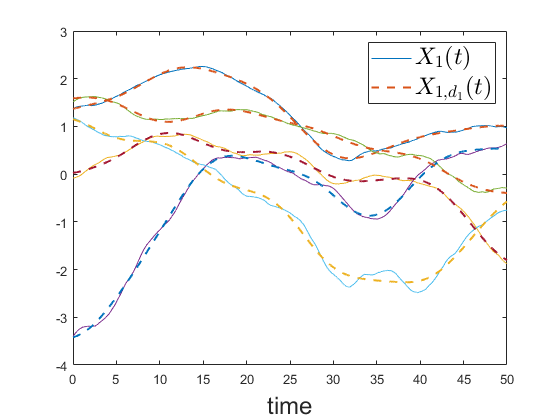}
  \hspace{0.1in}
  \includegraphics[scale=0.35]{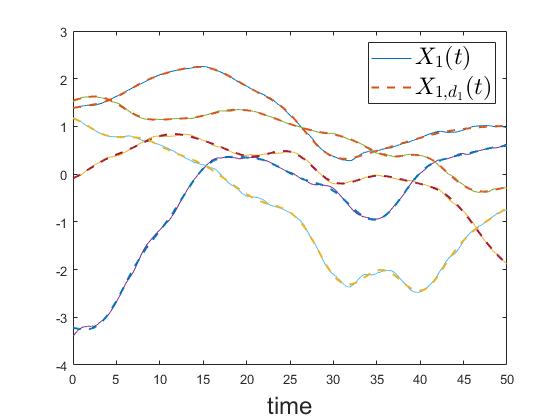}
  \caption{Five samples of $X_1(t)$ and $X_{1,d_1}(t)$ (solid and dotted line) for $d_1=5,10,15$ (left, middle and right panels).}
  \label{ex1-fig1}
\end{figure}

\begin{figure}[H]
  \centering
  \includegraphics[scale=0.35]{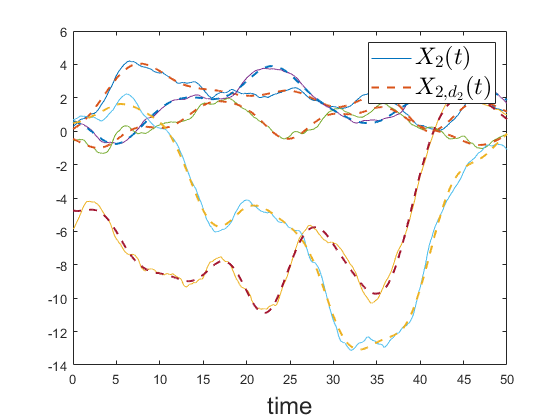}
  \hspace{0.1in}
  \includegraphics[scale=0.35]{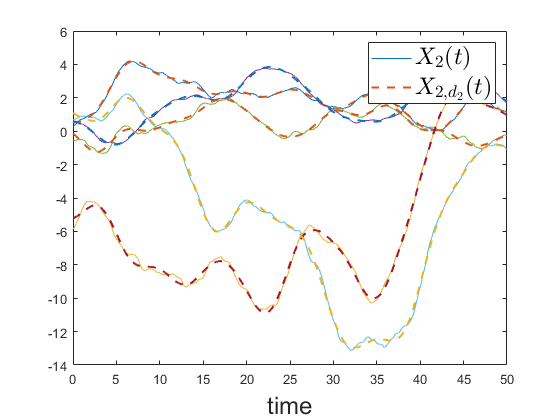}
  \hspace{0.1in}
  \includegraphics[scale=0.35]{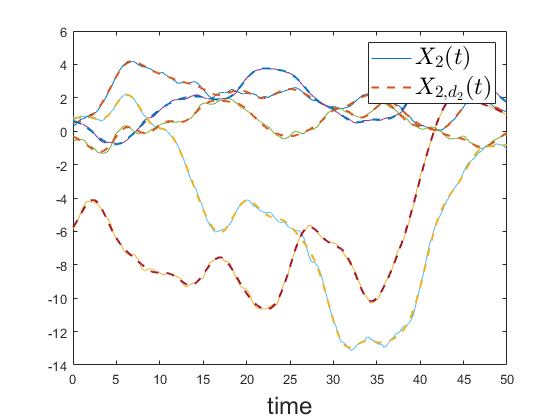}
  \caption{Five samples of $X_2(t)$ and $X_{2,d_2}(t)$ (solid and dotted line) for $d_2=15,20,25$ (left, middle and right panels).}
  \label{ex1-fig6}
\end{figure}


\begin{figure}[H]
  \centering
  \includegraphics[scale=0.35]{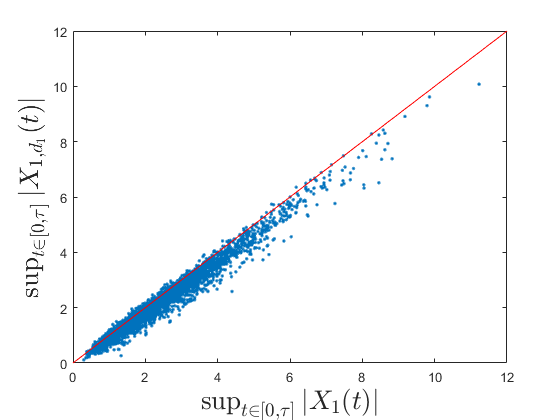}
  \hspace{0.1in}
  \includegraphics[scale=0.35]{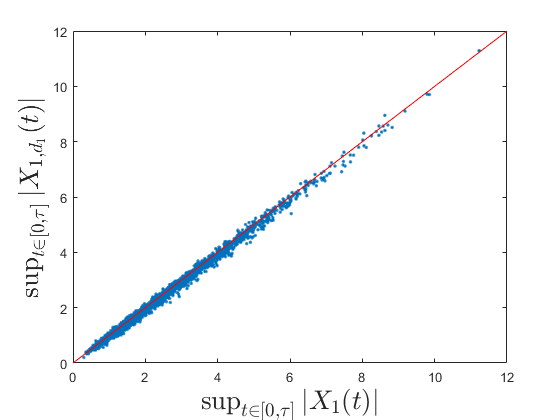}
  \hspace{0.1in}
  \includegraphics[scale=0.35]{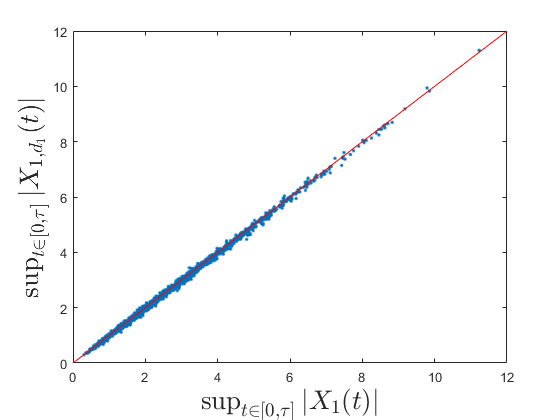}
  \caption{Scatter plots of $\sup_{t\in[0,\tau]}|X_1(t)|$ and $\sup_{t\in[0,\tau]}|X_{1,d_1}(t)|$ for $d_1=5,10,15$ (left, middle and right panels).}
 \label{ex1-fig4}
\end{figure}

\begin{figure}[H]
  \centering
  \includegraphics[scale=0.35]{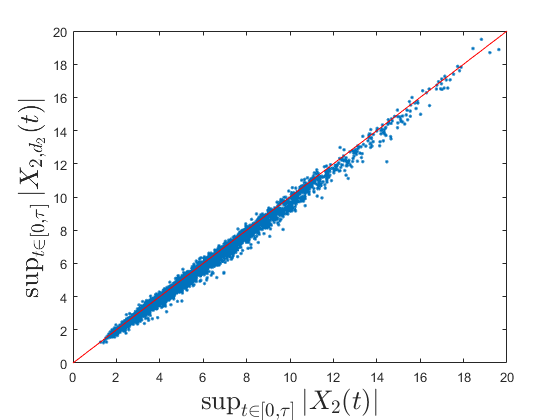}
  \hspace{0.1in}
  \includegraphics[scale=0.35]{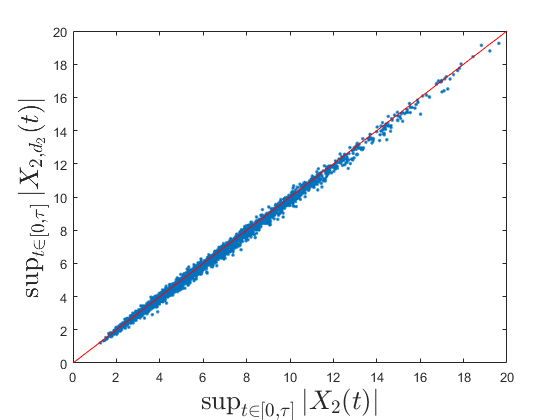}
  \hspace{0.1in}
  \includegraphics[scale=0.35]{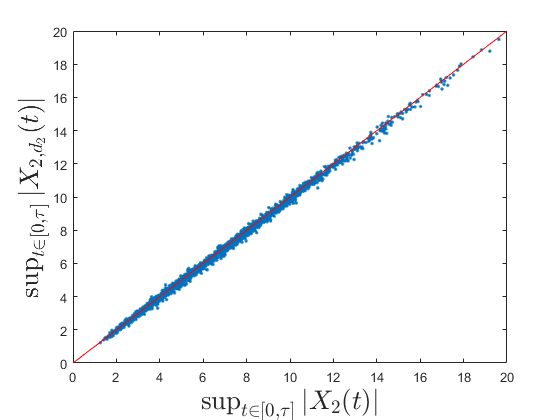}
  \caption{Scatter plots of $\sup_{t\in[0,\tau]}|X_2(t)|$ and $\sup_{t\in[0,\tau]}|X_{2,d_2}(t)|$ for $d_2=15,20,25$ (left, middle and right panels).}
  \label{ex1-fig5}
\end{figure}

\begin{figure}[H]
  \centering
  \includegraphics[scale=0.5]{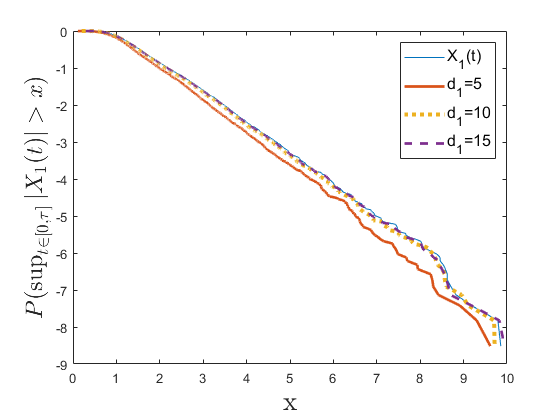}
  \hspace{0.1in}
  \includegraphics[scale=0.5]{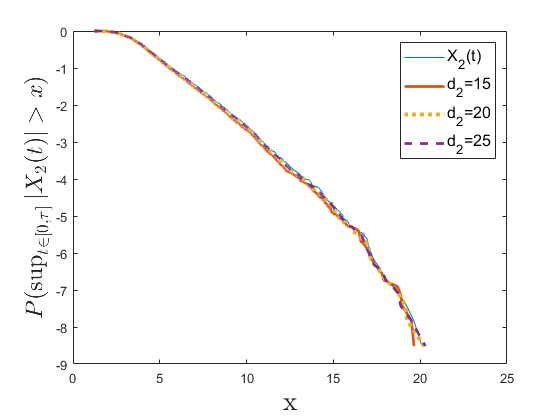}
  \caption{Estimates of the target probability $P(\sup_{t\in[0,\tau]}|X_i(t)|>x)$ (thin solid line), estimates based on FD model for different values of $d_i$ (heavy solid line, dotted line and dashed line) in logarithmic scale for $X_1(t)$ and $X_2(t)$ (left and right panels).}
  \label{ex1-fig7}
\end{figure}

The second FD model of $X(t)$ has the fom
\begin{eqnarray*}
X_{1,d_1}(t)=F_1^{-1}\circ\Phi(G_{1,d_1}(t)) \ \text{and} \ X_{2,d_2}(t)=F_1^{-1}\circ\Phi(G_{1,d_1}(t))+F_2^{-1}\circ\Phi(G_{2,d_2}(t)),
\end{eqnarray*}
where $G_{1,d_1}(t)$ and $G_{2,d_2}(t)$ are FD models of $G_1(t)$ and $G_2(t)$ given by (\ref{201b}) with $n=1$. According to Corollary \ref{cor1}, $\sup_{t\in[0,\tau]}|F_i^{-1}\circ\Phi(G_{i,d_i}(t))-F_i^{-1}\circ\Phi(G_{i}(t))|\overset{w}{\to}0$, as $\min\{d_1,d_2\}\to\infty$,
which implies $\sup_{t\in[0,\tau]}|X_{i,d}(t)-X_i(t)|\overset{w}{\to}0$, as $\min\{d_1,d_2\}\to\infty$.  This means that most of the FD samples will approximate accurately the corresponding target samples for a sufficiently large stochastic dimension.

The left, middle and right panels of Figs.~\ref{ex2-fig1} and \ref{ex2-fig6} show with solid and dotted lines five samples of $X_i(t)$ and $X_{i,d_i}(t)$, $i=1,2$ for $d_1=5,10,15$ and $d_2=15,20,25$. Scatter plots of $\big(\sup_{0\leq t\leq \tau}|X_i(t)|,\sup_{0\leq t\leq \tau}|X_{i,d_i}(t)|\big)$ are in Figs.~\ref{ex2-fig4} and \ref{ex2-fig5} for the same values of $d_i$, $i=1,2$ (left, middle and right panels).

The thin solid lines of the left and right panels of Fig.~\ref{ex2-fig7} are estimates of $P(\sup_{t\in[0,\tau]}|X_i(t)|$\\ $>x)$ for $i=1$ and $i=2$ which are obtained directly from data. These probabilities are viewed as reference. The other lines of the figure are calculated from FD models for $d_1=5$ and $d_2=15$ (heavy solid lines), $d_1=10$ and $d_2=20$ (dotted lines) and $d_1=15$ and $d_2=25$ (dashed lines) for the first and second components (left and right panels). The dashed lines are the closest to the reference.
These plots show, in agreement with our theoretical results, that the discrepancy between samples and extremes of $X_i(t)$ and $X_{i,d_i}(t)$ can be made as small as desired by increasing the stochastic dimension $d_i$, $i=1,2$. The FD models of $X_1(t)$ require a lower truncation level since the samples of $X_1(t)$ are smoother than those of $X_2(t)$, see Figs.~\ref{ex2-fig1} and \ref{ex2-fig6}.
\begin{figure}[H]
  \centering
  \includegraphics[scale=0.35]{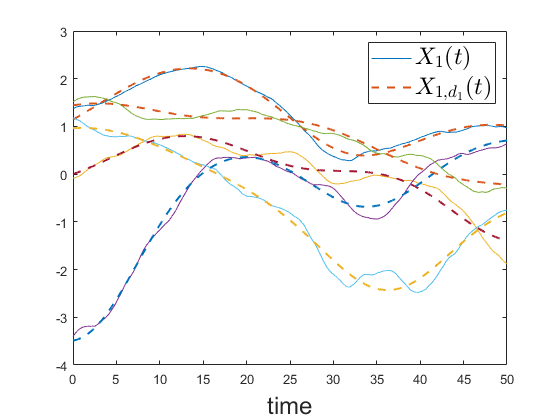}
  \hspace{0.1in}
  \includegraphics[scale=0.35]{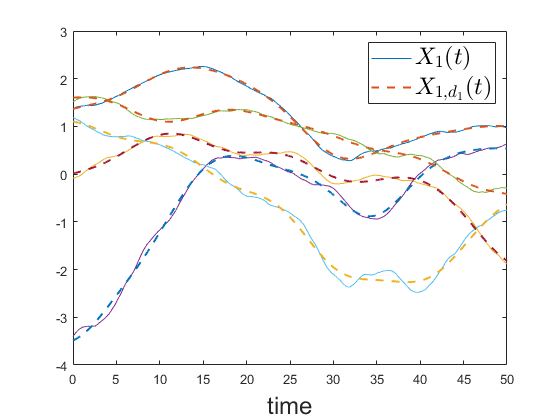}
  \hspace{0.1in}
  \includegraphics[scale=0.35]{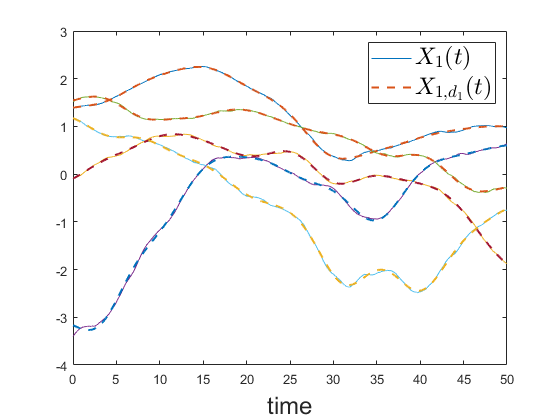}
  \caption{Five samples of $X_1(t)$ and $X_{1,d_1}(t)$ (solid and dotted line) for $d_1=5,10,15$ (left, middle and right panels).}
  \label{ex2-fig1}
\end{figure}

\begin{figure}[H]
  \centering
  \includegraphics[scale=0.35]{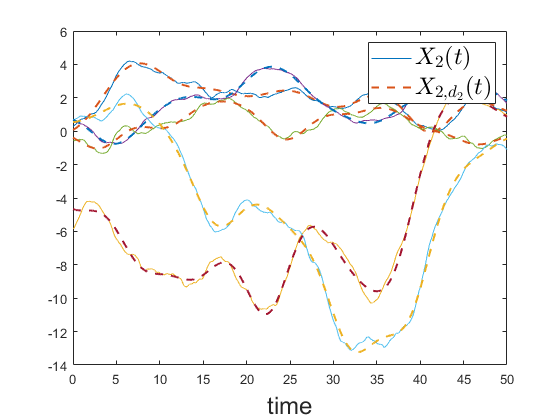}
  \hspace{0.1in}
  \includegraphics[scale=0.35]{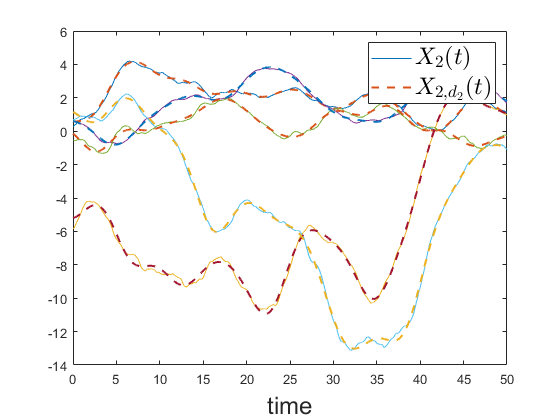}
  \hspace{0.1in}
  \includegraphics[scale=0.35]{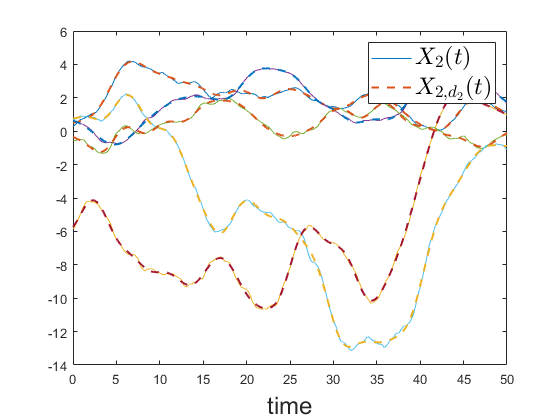}
  \caption{Five samples of $X_2(t)$ and $X_{2,d_2}(t)$ (solid and dotted line) for $d_2=15,20,25$ (left, middle and right panels).}
  \label{ex2-fig6}
\end{figure}

\begin{figure}[H]
  \centering
  \includegraphics[scale=0.35]{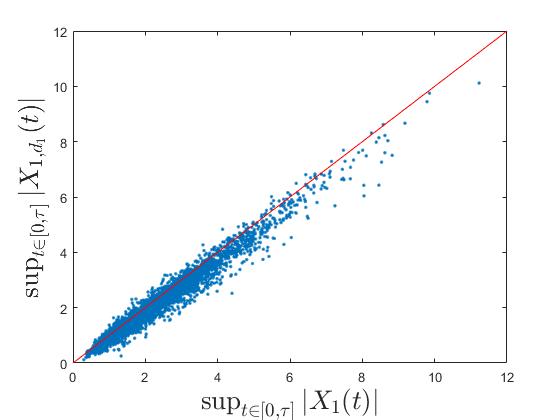}
  \hspace{0.1in}
  \includegraphics[scale=0.35]{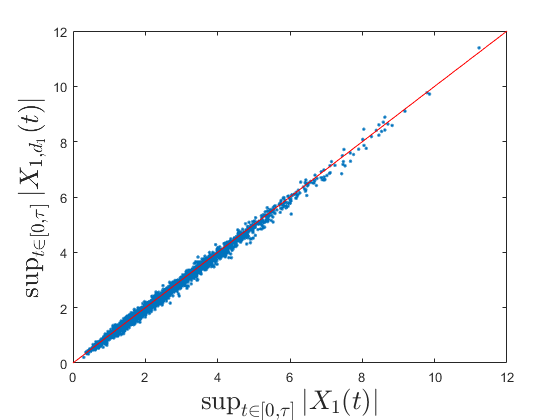}
  \hspace{0.1in}
  \includegraphics[scale=0.35]{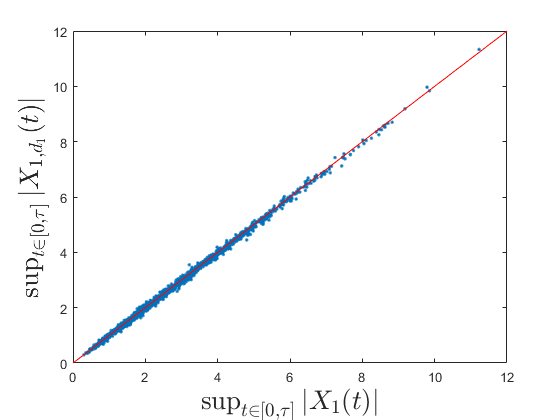}
  \caption{Scatter plots of $\sup_{t\in[0,\tau]}|X_1(t)|$ and $\sup_{t\in[0,\tau]}|X_{1,d_1}(t)|$ for $d_1=5,10,15$ (left, middle and right panels).}
 \label{ex2-fig4}
\end{figure}

\begin{figure}[H]
  \centering
  \includegraphics[scale=0.35]{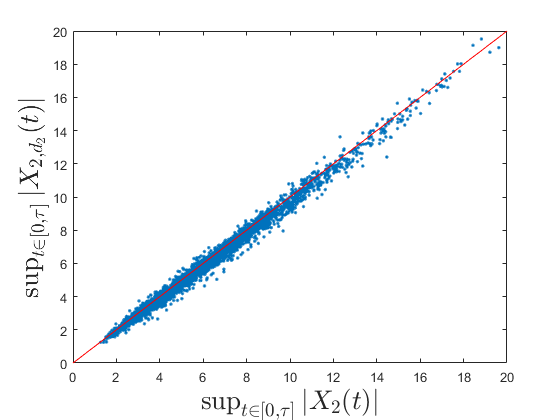}
  \hspace{0.1in}
  \includegraphics[scale=0.35]{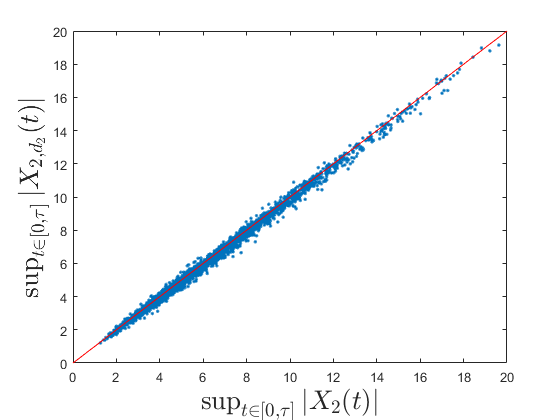}
  \hspace{0.1in}
  \includegraphics[scale=0.35]{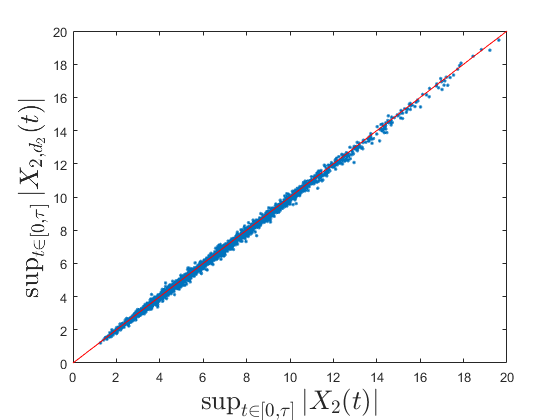}
  \caption{Scatter plots of $\sup_{t\in[0,\tau]}|X_2(t)|$ and $\sup_{t\in[0,\tau]}|X_{2,d_2}(t)|$ for $d_2=15,20,25$ (left, middle and right panels).}
  \label{ex2-fig5}
\end{figure}

\begin{figure}[H]
  \centering
  \includegraphics[scale=0.5]{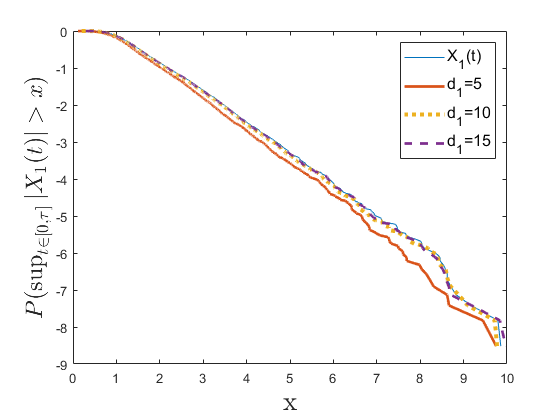}
  \hspace{0.1in}
  \includegraphics[scale=0.5]{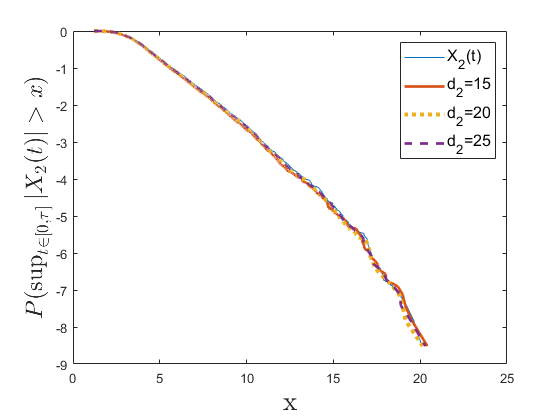}
  \caption{Estimates of the target probability $P(\sup_{t\in[0,\tau]}|X_i(t)|>x)$ (thin solid line), estimates based on FD model for different values of $d_i$ (heavy solid line, dotted line and dashed line) in logarithmic scale for $X_1(t)$ and $X_2(t)$ (left and right panels).}
  \label{ex2-fig7}
\end{figure}

We conclude this example with the observation that, according to the results of Figs.~\ref{ex1-fig4} to \ref{ex1-fig7} and Figs.~\ref{ex2-fig4} to \ref{ex2-fig7}, the performance of the two FD models of $X(t)$ is similar.

\begin{exam}\label{exam2}
{\rm Let $X(t)=(X_1(t),X_2(t))^T$, $0\leq t\leq \tau$, be an $\mathbb{R}^2$-valued stochastic process defined by the differential equations
\begin{eqnarray}\label{100}
\ddot{X}_1(t)+\alpha_1\dot{X}_1(t)+\beta_1X_1(t)&=&\gamma_1Y(t)^2, \nonumber\\
\ddot{X}_2(t)+\alpha_2\dot{X}_2(t)+\beta_2X_2(t)&=&\gamma_2Y(t)^2,
\end{eqnarray}
with the initial conditions $X_i(0)=0$ and $\dot{X}_i(0)=0$ for $i=1,2$, where $\alpha_i,\beta_i,\gamma_i>0$, $i=1,2$, are constants, $\beta_i-\alpha_i^2/4>0$, $Y(t)$ is the stationary solution of $dY(t)=-\rho \,Y(t)\,dt+\sqrt{2\,\rho} \,dB(t)$, $\rho>0$, and $B$ denotes the standard Brownian motion.
The following numerical results are for $\alpha_1=0.5$, $\alpha_2=0.2$, $\beta_1=10$, $\beta_2=5$, $\gamma_1=1$, $\gamma_2=2$, $\rho=1$, $\tau=10$ and the time step $\Delta t=0.01$.
The reported statistics are based on $5000$ samples of $X(t)$.}
\end{exam}

The processes $X_i(t)$, $\dot{X}_i(t)$ and $Y^2(t)$ have continuous samples as they are obtained from samples of $B(t)$ by integration and $B(t)$ has continuous samples.
We construct two types of FD models for $X(t)=(X_1(t),X_2(t))^T$. The first FD model is for the solution $X(t)$ of the differential equation (\ref{100}) by using  (\ref{201b}) with $n=2$.

Since the second derivatives of the correlation functions of $X_i(t)$ exist and are continuous in $[0,\tau]^2$, $\sup_{t\in[0,\tau]}|X_{i,d_i}(t)-X_i(t)|\overset{w}{\to}0$, as $d_i\to\infty$, $i=1,2$, see Theorem \ref{thm3}.  This suggests that most of the samples of $X(t)$ can be approximated accurately by FD samples and that the subset of samples of $X(t)$ which do not have this property decreases with the stochastic dimension of $X_d(t)$ .

The solid and dotted lines of Fig.~\ref{ex4-fig1}~ show a sample of $X(t)$ and $X_d(t)$ starting at the origin for $d_1=d_2=5$, $15$ and $25$ (left, middle and right panels). 
Scatter plots of $\big(\sup_{0\leq t\leq \tau}|X_i(t)|,\sup_{0\leq t\leq \tau}|X_{i,d_i}(t)|\big)$ are in Figs.~\ref{ex4-fig6} and \ref{ex4-fig7} for the same values of $d_i$, $i=1,2$ (left, middle and right panels). The thin solid lines of the left and right panels of Fig.~\ref{ex4-fig9} are estimates of $P(\sup_{t\in[0,\tau]}|X_i(t)|>x)$ for $i=1$ and $i=2$ which are obtained directly from data. These probabilities are viewed as reference. The other lines of the figure are calculated from FD models for $d_1=d_2=5$ (heavy solid lines), $d_1=d_2=10$ (dotted lines) and $d_1=d_2=15$ (dashed lines) for the first and second components (left and right panels). The dashed lines are the closest to the reference.
These plots show, in agreement with our theoretical results, that the discrepancy between samples and extremes of $X_i(t)$ and $X_{i,d_i}(t)$ can be made as small as desired by increasing the stochastic dimension $d_i$, $i=1,2$.

\begin{figure}[H]
  \centering
  \includegraphics[scale=0.35]{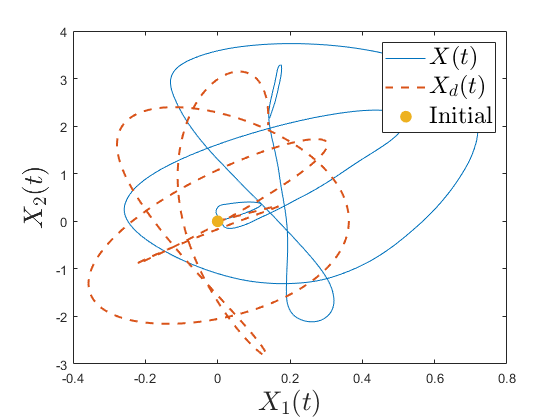}
  \hspace{0.1in}
  \includegraphics[scale=0.35]{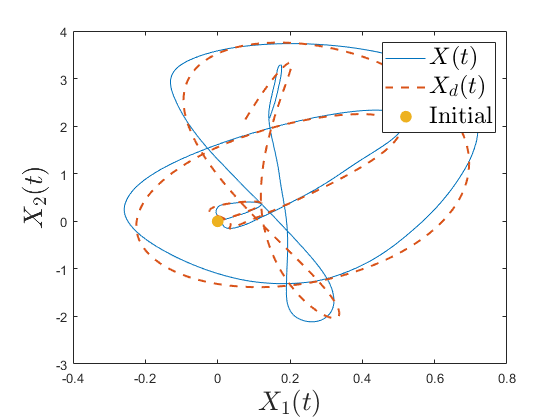}
  \hspace{0.1in}
  \includegraphics[scale=0.35]{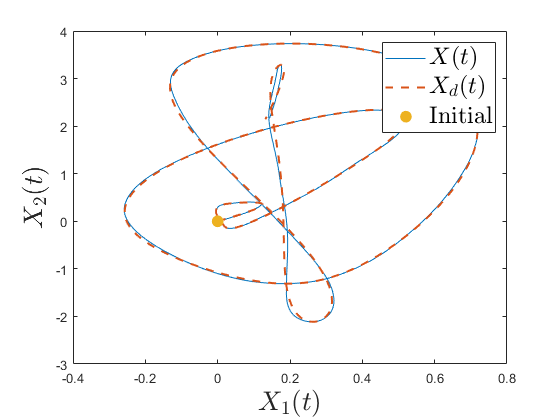}
  \caption{One sample of $(X_1(t),X_2(t))$ and $(X_{1,d_1}(t),X_{2,d_2}(t))$ (solid and dotted line) for $d_1=d_2=5,15,25$ (left, middle and right panels).}
  \label{ex4-fig1}
\end{figure}

\begin{figure}[H]
  \centering
  \includegraphics[scale=0.35]{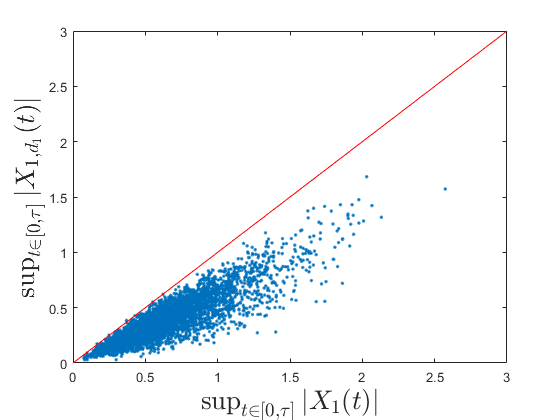}
  \hspace{0.1in}
  \includegraphics[scale=0.35]{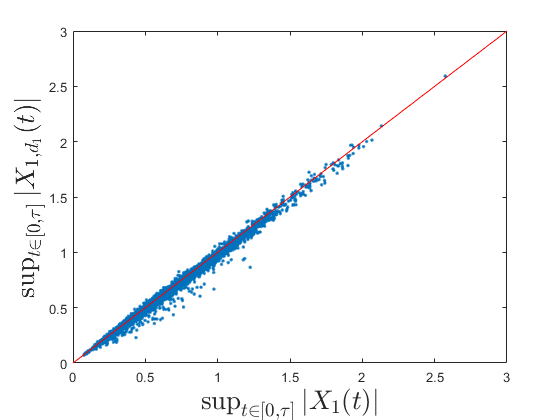}
  \hspace{0.1in}
  \includegraphics[scale=0.35]{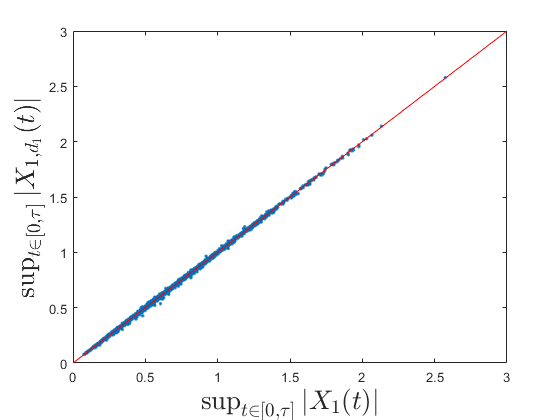}
  \caption{Scatter plots of $\sup_{t\in[0,\tau]}|X_1(t)|$ and $\sup_{t\in[0,\tau]}|X_{1,d_1}(t)|$ for $d_1=5,15,25$ (left, middle and right panels).}
  \label{ex4-fig6}
\end{figure}

\begin{figure}[H]
  \centering
  \includegraphics[scale=0.35]{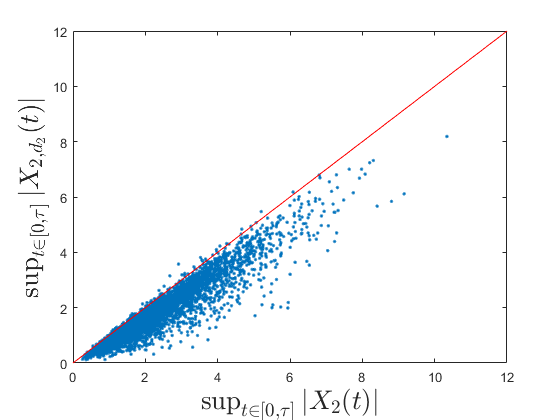}
  \hspace{0.1in}
  \includegraphics[scale=0.35]{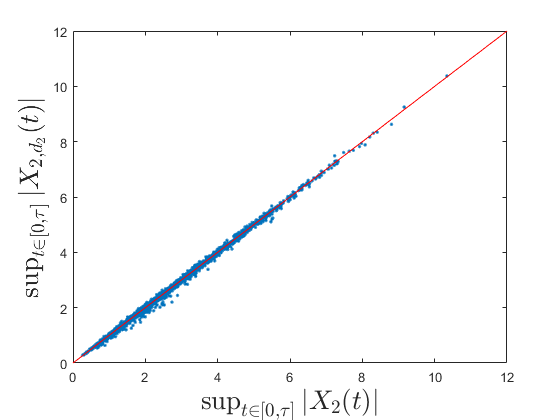}
  \hspace{0.1in}
  \includegraphics[scale=0.35]{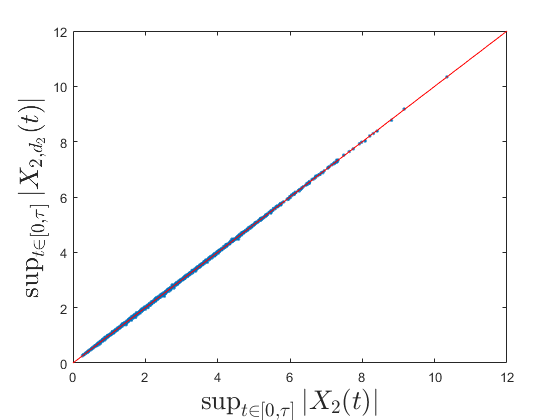}
  \caption{Scatter plots of $\sup_{t\in[0,\tau]}|X_2(t)|$ and $\sup_{t\in[0,\tau]}|X_{2,d_2}(t)|$ for $d_2=5,15,25$ (left, middle and right panels).}
  \label{ex4-fig7}
\end{figure}

\begin{figure}[H]
  \centering
  \includegraphics[scale=0.50]{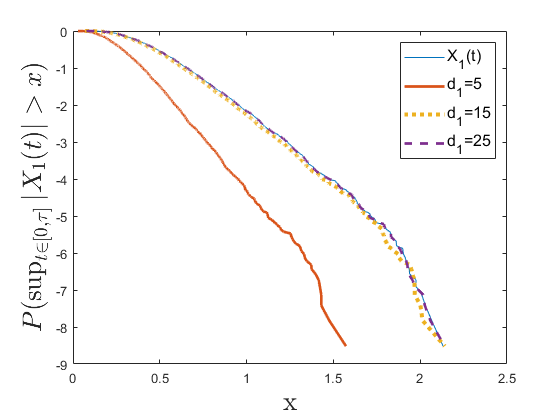}
  \hspace{0.1in}
  \includegraphics[scale=0.50]{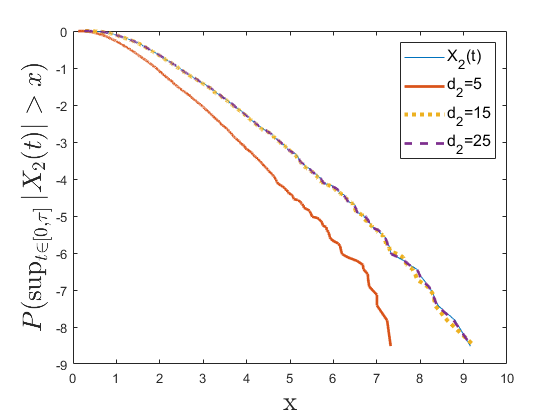}
  \caption{Estimates of the target probability $P(\sup_{t\in[0,\tau]}|X_i(t)|>x)$ (thin solid line), estimates based on FD model for different values of $d_i$ (heavy solid line, dotted line and dashed line) in logarithmic scale for $X_1(t)$ and $X_2(t)$ (left and right panels).}
  \label{ex4-fig9}
\end{figure}

The second FD model of $X(t)$ is constructed from  FD models of the input $Y(t)$.   We have
\begin{eqnarray}\label{solution1}
X_{i}(t)=\int_{0}^{t}\frac{\gamma_i}{\psi_i}e^{-\alpha_i(t-u)/2}\sin(\psi_i(t-u))Y(u)^2du, \ i=1,2, \quad 0\leq t\leq\tau,
\end{eqnarray}
where $\psi_i=(\beta_i-\alpha_i^2/4)^{1/2}$, $i=1,2$, \cite[Chap.~2]{mg-ld}. The FD models $Y_d(t)$ of $Y(t)$ are given by (\ref{201b}) with $n=1$, then the FD models of $X_d(t)=(X_{1,d}(t),X_{2,d}(t))^T$ result from (\ref{solution1}) with $Y_d(t)$ in place of $Y(t)$, i.e.,
\begin{eqnarray}\label{solution2}
X_{i,d}(t)=\int_{0}^{t}\frac{\gamma_i}{\psi_i}e^{-\alpha_i(t-u)/2}\sin(\psi_i(t-u))Y_d(u)^2du, \ i=1,2, \quad 0\leq t\leq\tau.
\end{eqnarray}
Since $Y(t)$ is a Gaussian process with continuous samples and continuous correlation function $c(s,t)=e^{-\rho|s-t|}$, we have $\sup_{t\in[0,\tau]}|Y_d(t)-Y(t)|\overset{w}{\to}0$ by Theorem \ref{thm4-1}. Therefore,
\begin{eqnarray*}
&&\sup_{t\in[0,\tau]}|X_{i,d}(t)-X_i(t)|\nonumber\\
&=&\sup_{t\in[0,\tau]}\bigg|\int_{0}^{t}\frac{\gamma_i}{\psi_i}e^{-\alpha_i(t-u)/2}\sin(\psi_i(t-u))\Big(Y_d(u)^2-Y(u)^2\Big)du\bigg|\nonumber\\
&\leq&\int_{0}^{\tau}\frac{\gamma_i}{\psi_i}|Y_d(u)^2-Y(u)^2|du\leq\frac{\gamma_i\tau}{\psi_i}\sup_{t\in[0,\tau]}|Y_d(u)^2-Y(u)^2|\overset{w}{\to}0.
\end{eqnarray*}
This suggest that most of the samples of $X(t)$ can be approximated accurately by FD samples and that the subset of samples of $X(t)$ which do not have this property decreases with the stochastic dimension of $X_d(t)$ .

The solid and dotted lines of Fig.~\ref{ex5-fig1}~ show a sample of $X(t)$ and $X_d(t)$ starting at the origin for $d_1=d_2=5$, $15$ and $25$ (left, middle and right panels). 
Scatter plots of $\big(\sup_{0\leq t\leq \tau}|X_i(t)|,\sup_{0\leq t\leq \tau}|X_{i,d_i}(t)|\big)$ are in Figs.~\ref{ex5-fig6} and \ref{ex5-fig7} for the same values of $d_i$, $i=1,2$ (left, middle and right panels). The thin solid lines of the left and right panels of Fig.~\ref{ex5-fig9} are estimates of $P(\sup_{t\in[0,\tau]}|X_i(t)|>x)$ for $i=1$ and $i=2$ which are obtained directly from data. These probabilities are viewed as reference. The other lines of the figure are calculated from FD models for $d_1=d_2=5$ (heavy solid lines), $d_1=d_2=10$ (dotted lines) and $d_1=d_2=15$ (dashed lines) for the first and second components (left and right panels). The dashed lines are the closest to the reference.
These plots show, in agreement with our theoretical results, that the discrepancy between samples and extremes of $X_i(t)$ and $X_{i,d_i}(t)$ can be made as small as desired by increasing the stochastic dimension $d_i$, $i=1,2$.

\begin{figure}[H]
  \centering
  \includegraphics[scale=0.35]{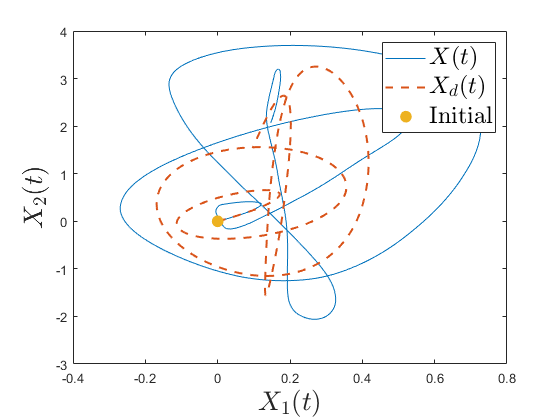}
  \hspace{0.1in}
  \includegraphics[scale=0.35]{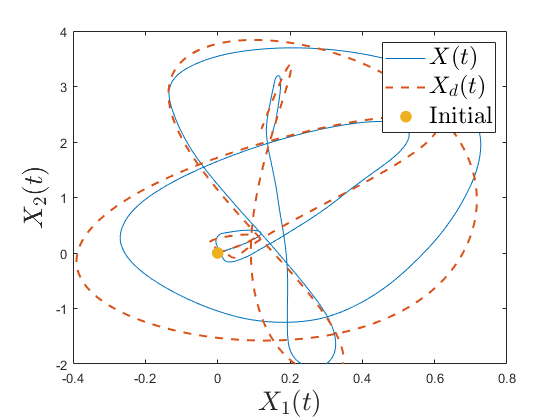}
  \hspace{0.1in}
  \includegraphics[scale=0.35]{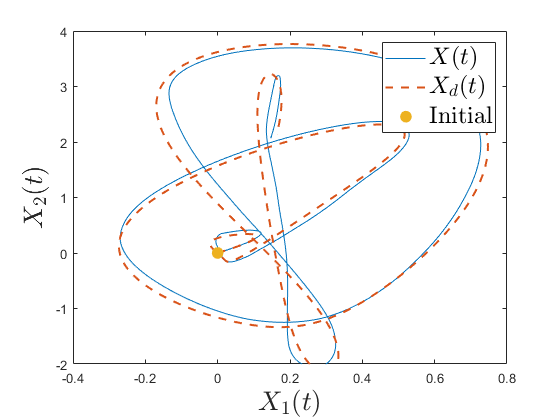}
  \caption{One sample of $(X_1(t),X_2(t))$ and $(X_{1,d}(t),X_{2,d}(t))$ (solid and dotted line) for $d=5,15,25$ (left, middle and right panels).}
  \label{ex5-fig1}
\end{figure}

\begin{figure}[H]
  \centering
  \includegraphics[scale=0.35]{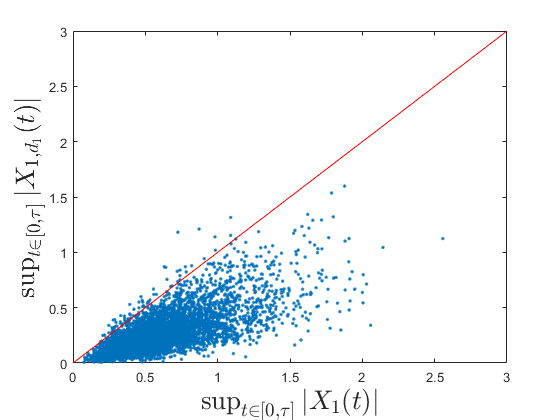}
  \hspace{0.1in}
  \includegraphics[scale=0.35]{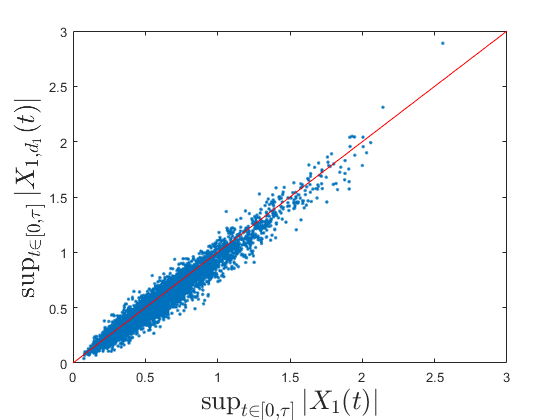}
  \hspace{0.1in}
  \includegraphics[scale=0.35]{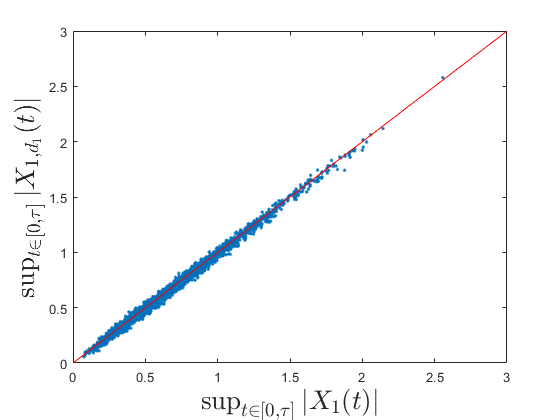}
  \caption{Scatter plots of $\sup_{t\in[0,\tau]}|X_1(t)|$ and $\sup_{t\in[0,\tau]}|X_{1,d}(t)|$ for $d=5,15,25$ (left, middle and right panels).}
  \label{ex5-fig6}
\end{figure}

\begin{figure}[H]
  \centering
  \includegraphics[scale=0.35]{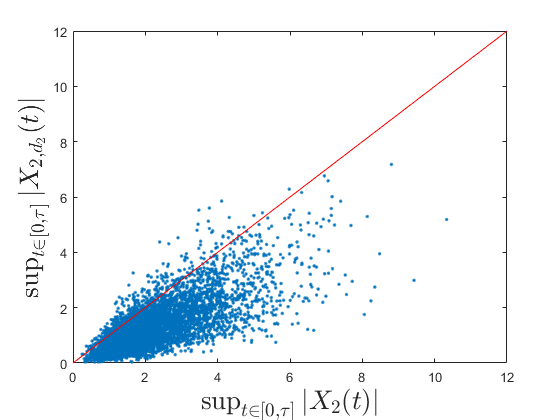}
  \hspace{0.1in}
  \includegraphics[scale=0.35]{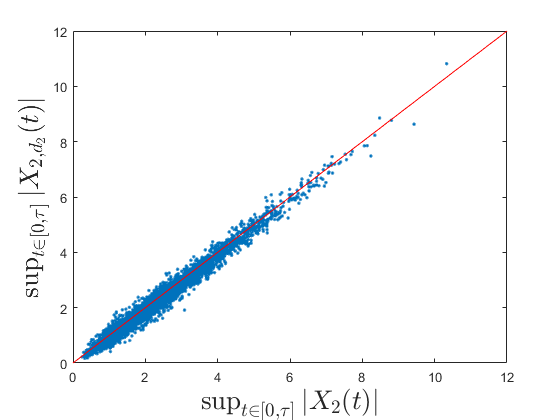}
  \hspace{0.1in}
  \includegraphics[scale=0.35]{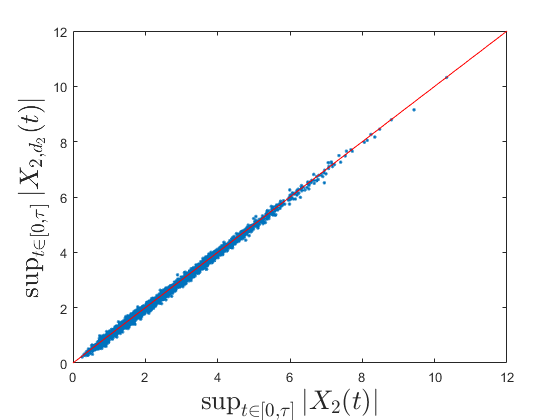}
  \caption{Scatter plots of $\sup_{t\in[0,\tau]}|X_2(t)|$ and $\sup_{t\in[0,\tau]}|X_{2,d}(t)|$ for $d=5,15,25$ (left, middle and right panels).}
  \label{ex5-fig7}
\end{figure}

\begin{figure}[H]
  \centering
  \includegraphics[scale=0.50]{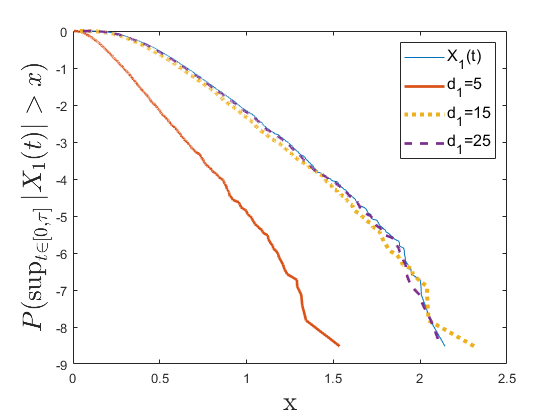}
  \hspace{0.1in}
  \includegraphics[scale=0.50]{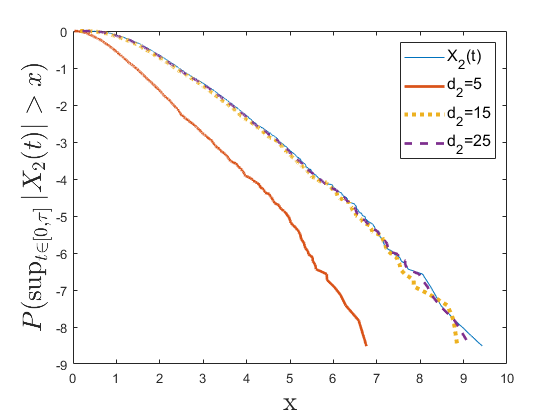}
  \caption{Estimates of the target probability $P(\sup_{t\in[0,\tau]}|X_i(t)|>x)$ (thin solid line), estimates based on FD model for different values of $d_i$ (heavy solid line, dotted line and dashed line) in logarithmic scale for $X_1(t)$ and $X_2(t)$ (left and right panels).}
  \label{ex5-fig9}
\end{figure}

Figures~\ref{ex4-fig6} to \ref{ex4-fig9} and Figs.~\ref{ex5-fig6} to \ref{ex5-fig9} suggest that the first FD models are superior in the sense that they approximate more accurately the extremes of the target processes. Note that the first FD model is constructed on the differentiable process $X(t)$, in contrast to the second FD model, which is constructed from the non-differentiable process $Y(t)$. The differentiability creates smoothness of the sample paths which may explain the superior performance of the first FD model.



\begin{exam}\label{exam3}
{\rm Let $U(t),t\in D=[0,\tau_1]\times[0,\tau_2]$, be the solution of the stochastic boundary value problem
\begin{eqnarray}\label{pde1}
\nabla\cdot(X(t,\omega)\nabla U(t,\omega))=0,\ t\in D
\end{eqnarray}
with the boundary conditions $U(0,t_2)=0$, $U(\tau_1,t_2)=1$, $\partial U(t_1,t_2)/\partial t_2=0$ for $t_2=0$ and $t_2=\tau_2$, a.s. Note that $t=(t_1,t_2)\in D$ is a space coordinate in this example. The random field $X(t)$ is defined by
\begin{eqnarray}\label{beta1}
X(t)=\alpha+(\beta-\alpha)F^{-1}_{{\rm Beta}(p,q)}\circ\Phi(G(t)), \quad t\in D,
\end{eqnarray}
where $0<\alpha<\beta<\infty$, $F_{{\rm Beta}(p,q)}$ denotes the distribution of a standard Beta random variable with shape parameters $(p,q)$ and $G(t),t\in\mathbb{R}^2$, is a homogeneous Gaussian field with zero-mean, unit variance, spectral density
\begin{eqnarray*}
s(\nu)=\frac{1}{2\pi\sqrt{1-\rho^2}}\exp\bigg\{-\frac{\nu_1^2-2\rho \nu_1\nu_2+\nu_2^2}{2(1-\rho^2)}\bigg\},\ \rho\in(-1,1), \ \nu=(\nu_1,\nu_2)\in\mathbb{R}^2
\end{eqnarray*}
and correlation function
\begin{eqnarray}\label{corr}
c(s,t)=\exp\bigg\{-\frac{(t_1-s_1)^2+2\rho(t_1-s_1)(t_2-s_2)+(t_2-s_2)^2}{2}\bigg\}, \ t, s\in\mathbb{R}^2.
\end{eqnarray}}
\end{exam}
Our objective is to estimate the distribution of the apparent conductivity \cite[Chap.7]{G2005,H2007,M2007}
\begin{eqnarray}\label{app1}
X_{\rm{app}}=\frac{1}{\tau_2}\int_{0}^{\tau_2}\int_{0}^{\tau_1}X(t)\frac{\partial U(t)}{\partial t_1}dt_1dt_2.
\end{eqnarray}
This objective cannot be achieved directly since the distribution of $X_{\rm{app}}$ is not available analytically and it is not possible to generate samples of this random variable as the integrand in its definition has infinite stochastic dimension.

To characterize $X_{\rm{app}}$, we construct a family of surrogates $\{X_{\rm{app},d}\}$, $d=1,2,\cdots$, which have the following two properties. First, samples of $X_{\rm{app},d}$ can be generated by standard Monte Carlo algorithms. Second, the distribution of $X_{\rm{app},d}$ converges to that of $X_{\rm{app}}$ as $d\to\infty$. The surrogates $\{X_{\rm{app},d}\}$ are defined by
\begin{eqnarray}\label{app2}
X_{\rm{app},d}=\frac{1}{\tau_2}\int_{0}^{\tau_2}\int_{0}^{\tau_1}X_d(t)\frac{\partial U_d(t)}{\partial t_1}dt_1dt_2,
\end{eqnarray}
where $U_d(t)$ is the solution of the stochastic boundary value problem
\begin{eqnarray}\label{pde2}
\nabla\cdot(X_d(t,\omega)\nabla U_d(t,\omega))=0,\ t\in D
\end{eqnarray}
with the boundary conditions of $(\ref{pde1})$, $X_d(t)$ is defined by
\begin{eqnarray}\label{beta2}
X_d(t)=\alpha+(\beta-\alpha)F^{-1}_{{\rm Beta}(p,q)}\circ\Phi(G_d(t)), \quad t\in D
\end{eqnarray}
and $G_d(t)$ is the FD model of $G(t)$. The random field $G_d(t)$ has the form
\begin{equation}\label{ex3-1}
  G_{d}(t)=\sum_{k=1}^{d} Z_{k}\,\varphi_{k}(t), \quad d=1,2,\ldots, \quad t\in D,
\end{equation}
where $\varphi_{k}(t)$, $k=1,\cdots,d$ are the top $d$ eigenfunctions of $c(s,t)$, i.e., the eigenfunctions corresponding to the largest $d$ eigenvalues, and the random coefficients $\{Z_{k}\}$ are defined sample-by-sample from samples of $G(t)$ by projection, i.e.,
\begin{eqnarray}\label{ex3-2}
Z_{k}(\omega)=\int_{D}G(t,\omega)\varphi_{k}(t)dt,  \quad k=1,2,\cdots, \quad \omega\in\Omega.
\end{eqnarray}
Note that the random elements $X_d(t)$, $U_d(t)$ and $X_{\rm{app},d}$ depend on the finite set of random variables  $(Z_1,\cdots,Z_d)$ so that samples of $X_{\rm{app},d}$ can be generated by standard Monte Carlo algorithms.  We now show that $X_{\rm{app},d}$ also has the second property, i.e., its samples can be used to estimate the distribution of $X_{\rm app}$ provided that $d$ is sufficiently large.

From (\ref{pde1}) and (\ref{pde2}), we have
\begin{eqnarray*}
\nabla\cdot(X(t)\nabla U(t))=\nabla\cdot(X_d(t)\nabla U_d(t)),
\end{eqnarray*}
which implies that $U_d(t)-U(t)$ satisfies the stochastic equation
\begin{eqnarray}\label{pde3}
\nabla\cdot\Big(X(t)\nabla\big(U_d(t)-U(t)\big)\Big)=\xi_d(t),
\end{eqnarray}
with the boundary conditions $U(0,t_2)-U_d(0,t_2)=0$, $U(\tau_1,t_2)-U_d(\tau_1,t_2)=0$, $\partial (U(t_1,t_2)-U_d(t_1,t_2))/\partial t_2=0$ for $t_2=0$ and $t_2=\tau_2$, a.s., where
$\xi_d(t)=-\big(\nabla\big(X_d(t)-X(t)\big)\big)\cdot\nabla U_d(t)-\big(X_d(t)-X(t)\big)\Delta U_d(t)$.
The solution of (\ref{pde3}) is
\begin{eqnarray*}
U_d(t,\omega)-U(t,\omega)=\int_{D}g(t,s,\omega)\xi_d(s,\omega)ds, \quad \text{\rm a.s.},
\end{eqnarray*}
where $g(t,s,\omega)$ is the Green function of the operator $\nabla\cdot(X(t,\omega)\nabla(\cdot))$ \cite{CH1953} Section V.14.   The Green function $g(t,s,\omega)$ is bounded in $D$ for a fixed $\omega\in\Omega$ since it is continuous and $D$ is bounded. We assume that (1) there is a single finite bound for the entire family of Green functions, i.e., $\sup_{t,s\in D}|g(t,s)|$ is almost surely bounded, (2) $\sup_{t\in D}|U(t)|$, $\sup_{t\in D}|\partial U_d(t)/\partial t_i|$, $i=1,2$ and $\sup_{t\in D}|\Delta U_d(t)|$ are also almost surely bounded, and (3) the random fields $G(t)$ and $\partial G(t)/\partial t_i$, $i=1,2$ have continuous samples. Under these conditions,
\begin{eqnarray}\label{convergence1}
\sup_{t\in D}|U_d(t)-U(t)|&=&\bigg|\int_{D}g(t,s)\xi_d(s)ds\bigg|\nonumber\\
&\leq&\tau_1\tau_2\sup_{t,s\in D\times D}|g(t,s)|\sup_{s\in D}|\xi_d(s)|\overset{w}{\to}0, \ d\to\infty,
\end{eqnarray}
provided that $\sup_{s\in D}|\xi_d(s)|\overset{w}{\to}0$, as $d\to\infty$. Since
\begin{eqnarray*}
\sup_{t\in D}|\xi_d(t)|\leq \sum_{i=1}^{2}\sup_{t\in D}\bigg|\frac{\partial}{\partial t_i}X_d(t)-\frac{\partial}{\partial t_i}X(t)\bigg|\sup_{t\in D}\bigg|
\frac{\partial}{\partial t_i}U_d(t)\bigg|+\sup_{s\in D}|X_d(t)-X(t)|\sup_{t\in D}|\Delta U_d(t)|
\end{eqnarray*}
and $\sup_{t\in D}|X_{d}(t)-X(t)|\overset{w}{\to}0$ as $d\to\infty$ by Theorem \ref{thm4-1} and Corollary \ref{cor1} which also hold for random fields, it remains to show that
\begin{eqnarray}\label{convergence}
\sup_{t\in D}\bigg|\frac{\partial}{\partial t_i}X_d(t)-\frac{\partial}{\partial t_i}X(t)\bigg|\overset{w}{\to}0, \ i=1,2, \ d\to\infty.
\end{eqnarray}
Let $h(x)=\alpha+(\beta-\alpha)F^{-1}_{{\rm Beta}(p,q)}\circ\Phi(x)$, $x\in\mathbb{R}$, then
\begin{eqnarray*}
\sup_{t\in D}\bigg|\frac{\partial}{\partial t_i}X_d(t)-\frac{\partial}{\partial t_i}X(t)\bigg|
&=&\sup_{t\in D}\bigg|h^{\prime}(G_d(t))\frac{\partial}{\partial t_i}G_d(t)-h^{\prime}(G(t))\frac{\partial}{\partial t_i}G(t)\bigg|\nonumber\\
&\leq&\bigg(\sup_{t\in D}\bigg|\frac{\partial}{\partial t_i}G_d(t)\bigg|\bigg)\sup_{t\in D}|h^{\prime}(G_d(t))-h^{\prime}(G(t))|\nonumber\\
& &+\bigg(\sup_{t\in D}|h^{\prime}(G(t))|\bigg)\sup_{t\in D}\bigg|\frac{\partial}{\partial t_i}G_d(t)-\frac{\partial}{\partial t_i}G(t)\bigg|.
\end{eqnarray*}


It can be shown that $h^{\prime}(\cdot)$ is bounded and Lipschitz continuous by using L'H$\rm \hat{o}$pital's Rule. Since the Gaussian random fields $G(t)$ and $\partial G(t)/\partial t_i$, $i=1,2$ have continuous samples, and the correlation function $c(s,t)$ in (\ref{corr}) and its second derivative are continuous, Theorem \ref{thm4-1} applies to $G(t)$ and $\partial G(t)/\partial t_i$ so that $\sup_{t\in D}|G_d(t)-G(t)|\overset{w}{\to}0$ and $\sup_{t\in D}|(\partial G_d(t)/\partial t_i)-(\partial G(t)/\partial t_i)|\overset{w}{\to}0$, $i=1,2$, as $d\to\infty$. The Lipschitz continuity of $h^{\prime}(\cdot)$ and $\sup_{t\in D}|G_d(t)-G(t)|\overset{w}{\to}0$ as $d\to\infty$ imply the convergence $\sup_{t\in D}|h^{\prime}(G_d(t))-h^{\prime}(G(t))|\overset{w}{\to}0$, as $d\to\infty$.
Since $\sup_{t\in D}{\rm Var}[\partial G_d(t)/\partial t_i]<\infty$ for all $d\geq 1$, then $P(\sup_{t\in D}|\partial G_d(t)/\partial t_i|>x)\to0$, as $x\to\infty$ \cite{gs1991}, which implies $\sup_{t\in D}|\partial G_d(t)/\partial t_i|$, $i=1,2$, is almost surely bounded, $\sup_{t\in D}|h^{\prime}(G(t))|$ is also almost surely bounded since $h^{\prime}(\cdot)$ is bounded. Therefore, $\sup_{t\in D}|U_d(t)-U(t)|\overset{w}{\to}0$ as $d\to\infty$.



The discrepancy between the target and FD apparent conductivities can be bounded by
\begin{eqnarray*}
&&|X_{\rm{app},d}-X_{\rm{app}}|\nonumber\\
&=&\frac{1}{\tau_2}\bigg|\int_{0}^{\tau_2}\int_{0}^{\tau_1}\bigg(X_d(t)\frac{\partial U_d(t)}{\partial t_1}-
X(t)\frac{\partial U(t)}{\partial t_1}\bigg)dt_1dt_2\bigg|\nonumber\\
&=&\frac{1}{\tau_2}\bigg|\int_{0}^{\tau_2}\int_{0}^{\tau_1}\bigg(U_d(t)\frac{\partial X_d(t)}{\partial t_1}-
U(t)\frac{\partial X(t)}{\partial t_1}\bigg)dt_1dt_2+\int_{0}^{\tau_2}\Big(X_d(\tau_1,t_2)-X(\tau_1,t_2)\Big)dt_2\bigg|\nonumber\\
&\leq&\frac{1}{\tau_2}\int_{0}^{\tau_2}\int_{0}^{\tau_1}\bigg|\Big(U_d(t)-U(t)\Big)
\frac{\partial X_d(t)}{\partial t_1}\bigg|dt_1dt_2+\frac{1}{\tau_2}\int_{0}^{\tau_2}\int_{0}^{\tau_1}\bigg|U(t)\bigg(\frac{\partial X_d(t)}{\partial t_1}-
\frac{\partial X(t)}{\partial t_1}\bigg)\bigg|dt_1dt_2\nonumber\\
& &+\sup_{t_2\in[0,\tau_2]}|X_d(\tau_1,t_2)-X(\tau_1,t_2)|\nonumber\\
&\leq&\tau_1\sup_{t\in D}|U_d(t)-U(t)|\sup_{t\in D}\bigg|\frac{\partial X_d(t)}{\partial t_1}\bigg|+\tau_1\sup_{t\in D}|U(t)|\sup_{t\in D}\bigg|\frac{\partial X_d(t)}{\partial t_1}-
\frac{\partial X(t)}{\partial t_1}\bigg|\nonumber\\
& &+\sup_{t\in D}|X_d(t)-X(t)|
\end{eqnarray*}
by using boundary conditions, i.e., $U(0,t_2)=U_d(0,t_2)=0$ and $U(\tau_1,t_2)=U_d(\tau_1,t_2)=1$, and integration by parts.

The above inequality implies the convergence $|X_{\rm{app},d}-X_{\rm{app}}|\overset{w}{\to}0$ as $d\to\infty$ since
$(i)$~$\sup_{t\in D}|U(t)|<\infty$ by assumption, $(ii)$~$\sup_{t\in D}|U_d(t)-U(t)|\overset{w}{\to}0$, $\sup_{t\in D}|X_d(t)-X(t)|\overset{w}{\to}0$ and $\sup_{t\in D}|(\partial X_d(t)/\partial t_1)-(\partial X(t)/\partial t_1)|\overset{w}{\to}0$ as $d\to\infty$, a property already shown, and $(iii)$~ $\sup_{t\in D}|\partial X_d(t)/\partial t_1|\leq \sup_{t\in D}|h^{\prime}(G_d(t))|\sup_{t\in D}|\partial G_d(t)/\partial t_1|$, where $\sup_{t\in D}|\partial G_d(t)/\partial t_1|<\infty$ almost surely is already shown, $\sup_{t\in D}|h^{\prime}(G_d(t))|<\infty$ almost surely holds by the boundedness of $h^{\prime}(\cdot)$.

The following numerical results are for $\alpha=1$, $\beta=20$, $p=0.5$, $q=1.5$, $\rho=0.7$, $\tau_1=20$ and $\tau_2=15$. The reference conductivity samples, i.e., the samples of $X_{\rm app}$ have been calculated from the corresponding samples of a discrete version of $G(t)$ corresponding to a mesh with sides $\Delta t_1=0.4$ and $\Delta t_2=0.3$. The reported statistics are based on $1000$ samples of $X_{\rm{app}}$. Figure~\ref{ex6-fig1}~ shows a sample of $G(t)$ and $G_d(t)$ for $d=50,150$ (left, middle and right panels). Figure~\ref{ex6-fig2}~ shows the corresponding samples of $X(t)$ and $X_d(t)$ for $d=50,150$ (left, middle and right panels) obtained from (\ref{beta1}) and (\ref{beta2}). Figure~\ref{ex6-fig3}~ shows the corresponding samples of $U(t)$ and $U_d(t)$ for $d=50,150$ (left, middle and right panels). The left and middle panels of Fig.~\ref{ex6-fig5} show the scatter plots of $X_{\rm{app}}$ and $X_{\rm{app},d}$ for $d=50,150$.

The thin solid line in the right panel of Fig.~\ref{ex6-fig5} is an estimate of $P(X_{\rm{app}}>x)$ which is obtained from (\ref{app1}) by using the reference samples of $X(t)$ and $G(t)$. It is viewed as reference. The other lines of the figure are calculated from FD models for $d=50$ (heavy solid line) and $d=150$ (dotted line). The dotted line is the closest to the reference. These plots show, in agreement with our theoretical results, that the discrepancy between samples of $G(t)$ and $G_{d}(t)$, $X(t)$ and $X_{d}(t)$, $U(t)$ and $U_{d}(t)$, and distributions of $X_{\rm{app}}$ and $X_{\rm{app},d}$ can be made as small as desired by increasing the stochastic dimension $d$.

\begin{figure}[H]
  \centering
  \includegraphics[scale=0.35]{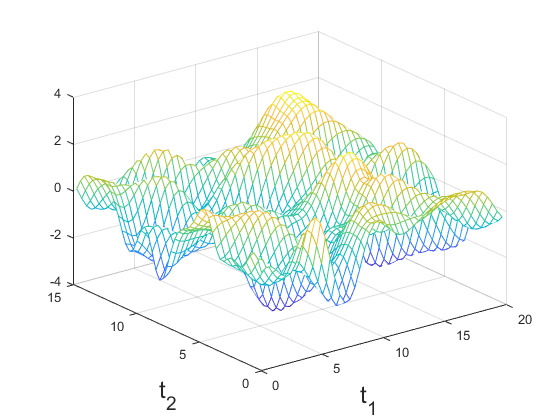}
  \hspace{0.1in}
  \includegraphics[scale=0.35]{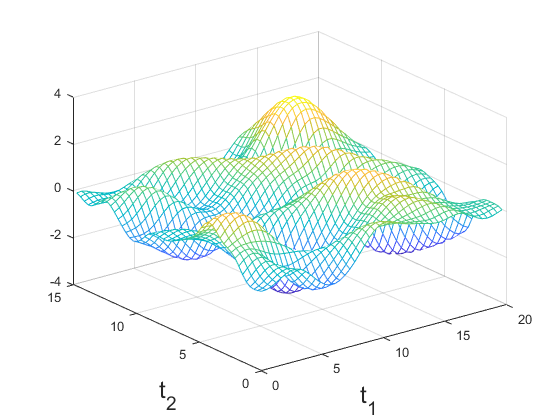}
  \hspace{0.1in}
  \includegraphics[scale=0.35]{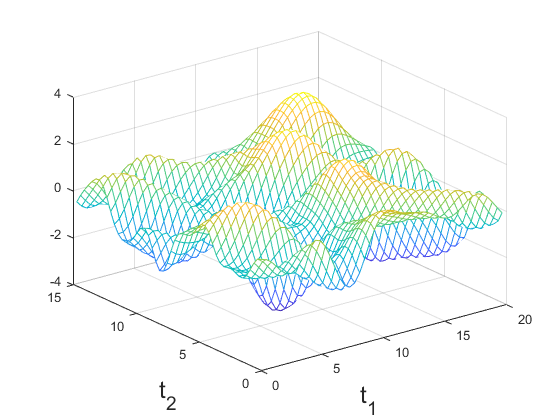}
  \caption{One sample of $G(t)$ (left panel) and $G_d(t)$ for $d=50,150$ (middle and right panels).}
  \label{ex6-fig1}
\end{figure}

\begin{figure}[H]
  \centering
  \includegraphics[scale=0.35]{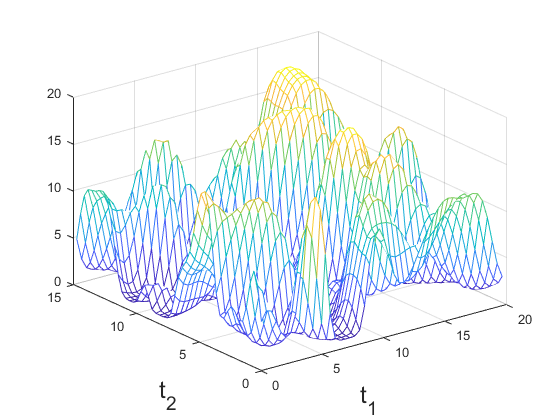}
  \hspace{0.1in}
  \includegraphics[scale=0.35]{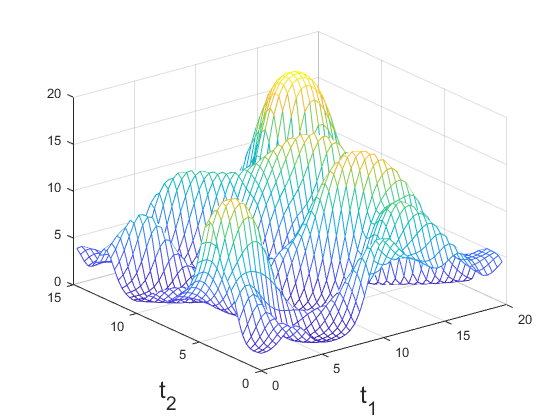}
  \hspace{0.1in}
  \includegraphics[scale=0.35]{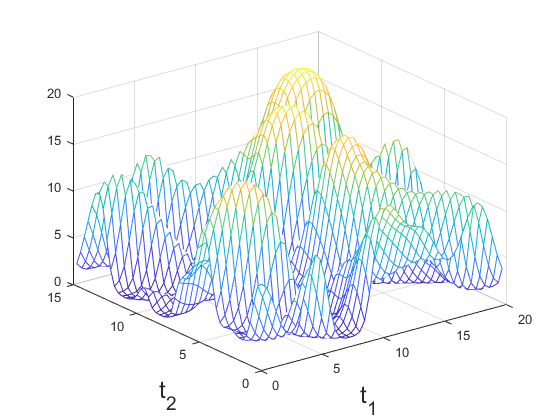}
  \caption{One sample of $X(t)$ (left panel) and $X_d(t)$ for $d=50,150$ (middle and right panels).}
  \label{ex6-fig2}
\end{figure}

\begin{figure}[H]
  \centering
  \includegraphics[scale=0.35]{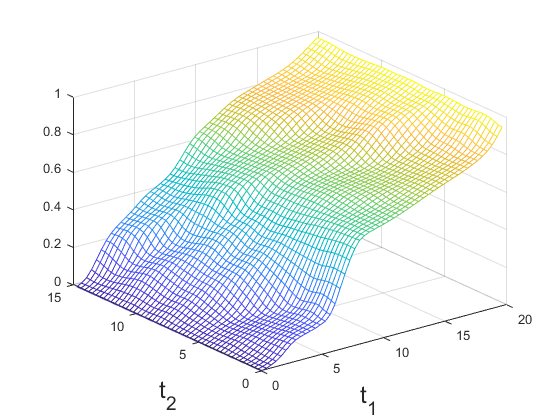}
    \hspace{0.1in}
  \includegraphics[scale=0.35]{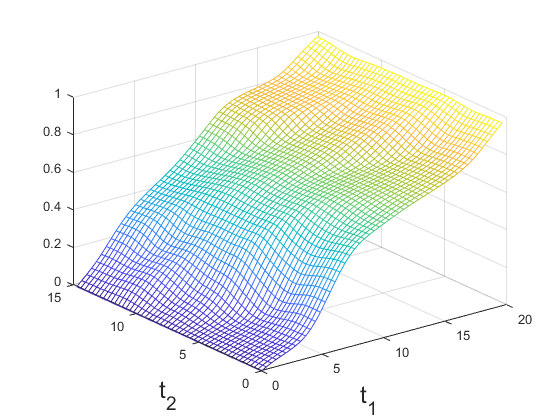}
  \hspace{0.1in}
  \includegraphics[scale=0.35]{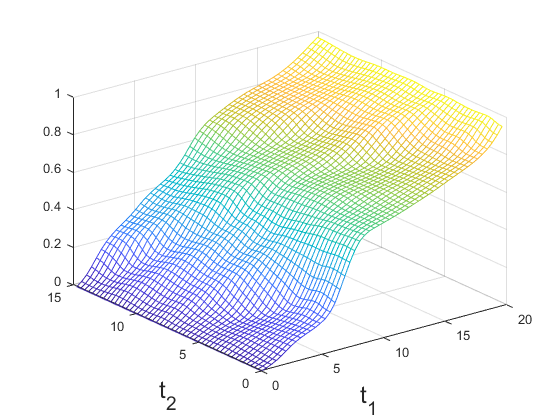}
  \caption{One sample of $U(t)$ (left panel) and $U_d(t)$ for $d=50,150$ (middle and right panels).}
  \label{ex6-fig3}
\end{figure}

\begin{figure}[H]
  \centering
  \includegraphics[scale=0.35]{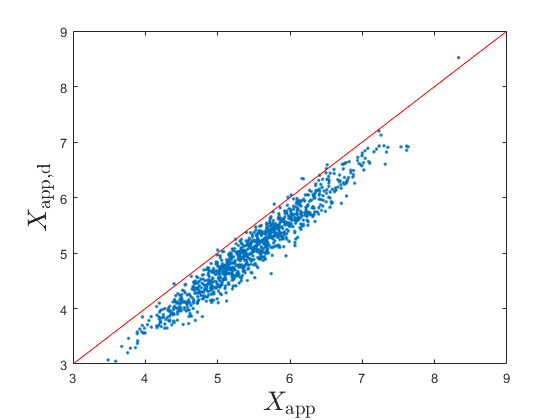}
    \hspace{0.1in}
  \includegraphics[scale=0.35]{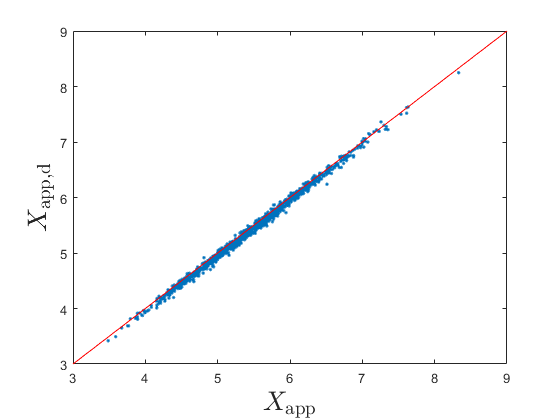}
  \hspace{0.1in}
  \includegraphics[scale=0.35]{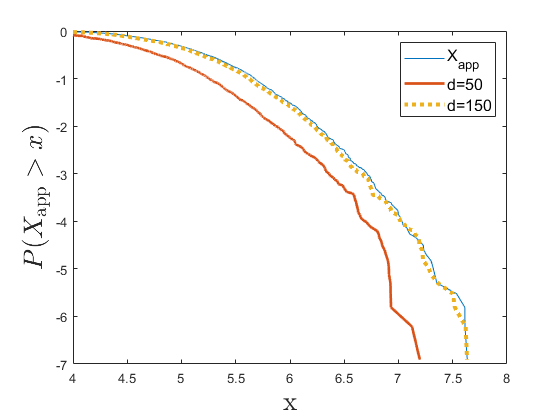}
  \caption{Scatter plots of $X_{\rm{app}}$ and $X_{\rm{app},d}$ for $d=50,150$ (left and middle panels). Estimates of the target probability $P(X_{\rm{app}}>x)$ (thin solid line), estimates based on FD model for $d=50$ (heavy solid line) and $d=150$ (dotted line) in logarithmic scale (right panel).}
  \label{ex6-fig5}
\end{figure}

\section{Conclusions}
\setcounter{thm}{0}\setcounter{Corol}{0}\setcounter{lemma}{0}\setcounter{pron}{0}\setcounter{equation}{0}
\setcounter{remark}{0}\setcounter{exam}{0}\setcounter{property}{0}\setcounter{defin}{0}

Finite dimensional (FD) models, i.e., deterministic functions of time/space and finite sets of random variables, have been constructed for target vector-valued random processes and real-valued random fields.  Samples of these models, referred to as FD samples, can be generated by standard Monte Carlo algorithms since, in contrast to target random functions which are uncountable families of random variables, as they depend on finite sets of random variables.  Conditions have been established under which FD samples can be used as surrogates for target samples and can be employed to estimate distributions of functionals of target random functions, e.g., distributions of extremes of these functions.

Numerical examples are presented to illustrate the implementation of FD models and show that the distributions of extremes of target random functions can be estimated from FD samples provided that they satisfy the conditions of our theorems.  The numerical examples include estimates of the distributions of extremes of two vector-valued non-Gaussian processes and of the distribution of the apparent property of a two dimensional material specimen with random properties.  It is shown that FD estimates of extremes and other functionals of target random functions are accurate for FD models satisfying the conditions of our theorems.\\

\section{Declarations}
\setcounter{thm}{0}\setcounter{Corol}{0}\setcounter{lemma}{0}\setcounter{pron}{0}\setcounter{equation}{0}
\setcounter{remark}{0}\setcounter{exam}{0}\setcounter{property}{0}\setcounter{defin}{0}

\noindent{\bf Funding.} The work reported in this paper has been partially supported by the National Science Foundation under the grant CMMI-2013697. This support is gratefully acknowledged.\\

\noindent{\bf Conflicts of interest.} The authors declare that they have no known competing financial interests or personal relationships that could have appeared to influence the work reported in this paper.\\

\noindent{\bf Data availability.} The datasets generated during the current study are not publicly available, since they have been generated for particular applications, but are available from the corresponding author on reasonable request.\\

\noindent{\bf Authors' contributions.} Hui Xu and Mircea D. Grigoriu confirm sole responsibility for the following: study conception and
design, data collection, analysis and interpretation of results, and manuscript
preparation

\end{document}